\documentclass[]{amsart}

\def\bn{\mathbb{N}}
 
\def\bc{\mathbb{C}} 
\def\bd{\mathbb{D}}
\def\bi{\mathbb{I}}
\def\bj{\mathbb{J}}
\def\h{\mathcal{H}} 
\def\k{\mathcal{K}} 
\def\l{\mathcal{L}}
\newcommand\cs{\mathcal{S}}
\newcommand\crr{\mathcal{R}}

\newcommand{\com}[1]{#1^{\prime}}
\newcommand{\inner}[2]{\langle #1,#2\rangle} 

\newcommand{\sumjn}{\sum_{j=1}^{n}}

\newcommand{\bh}{{\rm B}(\mathcal{H})}
\newcommand{\bhp}{{\rm B}(\mathcal{H}_{\pi})}
\newcommand{\bk}{{\rm B}(\mathcal{K})}
\newcommand{\ch}{{\rm C}(\mathcal{H})}

\newcommand{\trh}{{\rm T}(\mathcal{H})}

\newcommand{\trht}{{\rm T}(\mathcal{H}_2)}

\newcommand{\kh}{{\rm K}(\mathcal{H})}

\newcommand{\dxa}{D_{\xi}(a)}
\newcommand{\dxaa}{D_{\xi}(a^*)}
\newcommand{\dxbb}{D_{\xi}(b^*)}
\newcommand{\dxb}{D_{\xi}(b)}
\newcommand{\doa}{D_{\omega}(a)}
\newcommand{\doaa}{D_{\omega}(a^*)}
\newcommand{\dob}{D_{\omega}(b)}
\newcommand{\dobb}{D_{\omega}(b^*)}
\newcommand{\dis}[2]{D_{#1}(#2)}

\newcommand{\weakc}[1]{\overline{#1}}
\newcommand{\normc}[1]{\overline{\overline{#1}}}
\newcommand{\sap}[1]{\sigma_{\rm ap}(#1)}
\newcommand{\matn}[1]{{\rm M}_n(#1)}
\newcommand{\mattwo}[1]{{\rm M}_2(#1)}
\newcommand{\matm}[1]{{\rm M}_m(#1)}
\newcommand{\ort}[1]{#1^{\perp}}

\newcommand{\rea}[1]{{\rm Re}\,{#1}}
\newcommand{\xieta}{\xi\otimes\eta^*}
\newcommand{\xixi}{\xi\otimes\xi^*}
\newcommand{\du}[1]{#1^{\sharp}}
\newcommand{\bdu}[1]{#1^{\sharp\sharp}}
\newcommand{\dbar}{\overline{\partial}}
\newtheorem{theorem}{Theorem}[section]
\newtheorem{lemma}[theorem]{Lemma}
\newtheorem{co}[theorem]{Corollary}
\newtheorem{pr}[theorem]{Proposition}
\theoremstyle{remark}
\newtheorem{re}[theorem]{Remark}
\theoremstyle{definition}
\newtheorem{de}[theorem]{Definition}

\numberwithin{equation}{section}

\begin{document}

\title[]{Variance of operators and derivations} 
\author{Bojan Magajna} 
\address{Department of Mathematics\\ University of Ljubljana\\
Jadranska 21\\ Ljubljana 1000\\ Slovenia}
\email{Bojan.Magajna@fmf.uni-lj.si}

\thanks{Acknowledgment. I am grateful to Miran \v Cerne for discussions concerning complex anaysis topics, 
to Matej Bre\v sar and \v Spela \v Spenko for
conversations from which some of the questions studied in this paper have emerged, and to 
Victor Shulman for the correspondence concerning a question about Besov spaces.  I am especially grateful to the anonymous referee for all his comments and corrections of the paper.}

\thanks{The author was supported in part by the Ministry of Science
and Education of Slovenia.}



\begin{abstract}The variance of a bounded linear operator $a$ on a Hilbert space $\h$ at a unit vector $\xi$ is defined 
by $\dxa=\|a\xi\|^2-|\inner{a\xi}{\xi}|^2$. We show that two operators $a$ and $b$ have the same variance at all vectors 
$\xi\in\h$ if and
only if there exist scalars $\sigma,\lambda\in\bc$ with $|\sigma|=1$ such that $b=\sigma a+\lambda1$ or $a$ is normal and
$b=\sigma a^*+\lambda1$. Further, if $a$ is normal, then the inequality $\dxb\leq\kappa \dxa$ holds for some constant 
$\kappa$ and
all unit vectors $\xi$ if and only if $b=f(a)$ for a Lipschitz function $f$ on the spectrum of $a$.
Variants of these results for C$^*$-algebras are also proved, where vectors are replaced by pure states.

We also study the related inequalities $\|bx-xb\|\leq
\|ax-xa\|$ supposed to hold for all $x\in\bh$ or for all $x\in{\rm B}(\h^n)$ and all $n\in\bn$.  
We consider the connection between such inequalities and the range inclusion
$d_b(\bh)\subseteq d_a(\bh)$, where $d_a$ and $d_b$ are the derivations on $\bh$ induced by $a$ and $b$. 
If $a$ is subnormal, we study these conditions
in particular in the case when $b$ is of the form $b=f(a)$ for a function $f$.

\end{abstract}

\maketitle

\section{Introduction and notation}

The expected value of a quantum mechanical quantity represented by a selfadjoint operator $a$ on a complex Hilbert space
$\h$ in a state
$\omega$  is $\omega(a)$, while the {\em variance} of
$a$  is defined by $\doa=\omega(a^*a)-|\omega(a)|^2$. If $a$ is the multiplication by
a bounded measurable function on $L^2(\mu)$ for a probability measure $\mu$ and $\omega$ is the state
$x\mapsto\inner{x1}{1}$, where $1\in L^2(\mu)$ is the constant function, these notions reduce to the classical
notions of probability calculus. We may define the variance by the same formula for all
(not necessarily selfadjoint) operators $a\in\bh$. For a general vector state $\omega(a):=\inner{a\xi}{\xi}$, 
coming from a unit vector
$\xi\in\h$, the variance $\doa=\|a\xi\|^2-|\inner{a\xi}{\xi}|^2$ means just the square of the distance of 
$a\xi$ to the set of all scalar multiples of $\xi$. (Thus  $\doa=\eta_a(\xi)^2$, where $\eta$ is the 
function considered by Brown and Pearcy in \cite{BP}.)  We will prove that an operator $a$ is almost determined by its
variances: if $a,b\in\bh$ are such that $\doa=\dob$ for all vector states $\omega$
then $b=\alpha a+\beta1$ or $a$ is normal and $b=\alpha a^*+\beta1$ for some $\alpha,\beta\in\bc$ with $|\alpha|=1$
(Theorem \ref{th1}). We will also deduce a variant of this statement for C$^*$-algebras, where vector states are
replaced by pure states.

Then we will study the inequality 
\begin{equation}\label{-1}\dob\leq\kappa\doa,\end{equation} where $\kappa$ is a positive constant (which may be taken
to be $1$ if we replace $b$ by $\kappa^{-1/2}b$). If (\ref{-1}) holds for all vector states $\omega$,
then we will show that there exists a Lipschitz function $f:\sap{a}\to\sap{b}$, where $\sap{\cdot}$
denotes the approximate point spectrum, such that if $a$ is normal then $b=f(a)$ (Theorem \ref{th3}). For a general
$a$, however, $f$ is perhaps not nice enough to allow the definition of $f(a)$. Therefore we will also consider stronger variants of
(\ref{-1}).

For $2\times 2$ matrices (\ref{-1}) implies that $b=\alpha a+\beta1$ for some
scalars $\alpha,\beta\in\bc$ (Lemma \ref{le2}). But for general operators the condition  (\ref{-1}) 
is not very restrictive 
for it does not even imply that $b$ commutes with $a$. For example, if $a$ is hyponormal (\ref{-1}) holds with $b=a^*$ and $\kappa=1$. 
A simple 
computation (Lemma \ref{le41}) shows, 
however,  that for a vector state $\omega=\omega_{\xi}$ and a {\em normal} operator $a$ the quantity $\doa$ is just the square of the norm
of the operator $d_a(\xi\otimes\xi^*)$, where $d_a$ is the derivation on $\bh$, defined by
$d_a(x)=ax-xa$, and $\xi\otimes\xi^*$ is the rank one operator on $\h$, defined by $(\xi\otimes\xi^*)
\eta=\inner{\eta}{\xi}\xi$. Thus we will also study the  condition 
\begin{equation}\label{01}\|d_b(x)\|\leq\kappa\|d_a(x)\| \ \ (\forall x\in\bh),\end{equation}  where $a,b\in\bh$ and
$\kappa>0$ are fixed.  
We will show (Theorem \ref{th42}) that if equality holds in (\ref{01}) and $\kappa=1$ then  either
$b=\sigma a+\lambda1$ for some scalars $\sigma,\lambda\in\bc$ with $|\sigma|=1$ or there exist a unitary $u$ and 
scalars $\alpha,\beta,\lambda,\mu$ in $\bc$ with $|\beta|=|\alpha|$ such that $a=\alpha u^*+\lambda1$ 
and $b=\beta u+\mu1$. This will also be generalized to C$^*$-algebras.

For a normal operator $a$  Johnson and Williams  \cite{JW} proved that the condition  (\ref{01})
is equivalent to the range inclusion $d_b(\bh)\subseteq d_a(\bh)$. Their work was continued
by several researchers, including Williams \cite{W}, Fong \cite{F},  Kissin and Shulman \cite{KS}, Bre\v sar \cite{Br}   and in 
\cite{BMS} in different contexts, but still restricted to special classes of operators $a$ (such as normal,
isometric or algebraic). 
It is known that the range inclusion $d_b(\bh)\subseteq d_a(\bh)$ does not imply (\ref{01}) in general
since it does not even imply that $b$ is in the bicommutant $(a)^{\prime\prime}$ of $a$ \cite{Ho}. However the
author does not know of any  operators $a,b$ satisfying (\ref{01}) for which the
range inclusion does not hold. The corresponding purely algebraic problem for operators
on an (infinite dimensional) vector space $\mathcal{V}$, where $\bh$ is replaced by the algebra
${\rm L}(\mathcal{V})$ of all linear operators on $\mathcal{V}$ and the condition 
(\ref{01}) is replaced by the inclusion of the
kernels $\ker d_a\subseteq\ker d_b$, was studied in  \cite{M1}.

By the Hahn-Banach theorem the inclusion $d_b(\bh)\subseteq\normc{d_a(\bh)}$ (the norm closure) is equivalent to the requirement that
for each $\rho\in\du{\bh}$ (the dual of $\bh$) the condition $a\rho-\rho a=0$ implies $b\rho-\rho b=0$,
where $a\rho$ and $b\rho$ are functionals on $\bh$ defined by $(a\rho)(x)=\rho(xa)$ and $(\rho a)(x)=\rho(ax)$.
The operator spaces $\bh$ and $\du{\bh}$ are quite different (if $\h$ is infinite dimensional), so in general we can not 
expect a strong connection between (\ref{01}) and a formally similar condition
$$\label{03}\|b\rho-\rho b\|\leq\kappa\|a\rho-\rho a\|\ \ (\forall x\in\du{\bh}).$$
 
\smallskip
{\em Question.} Does (\ref{01}) imply at least
that the centralizer $C_a$ of $a$ in $\du{\bh}$ (that is, the set of all $\rho\in\du{\bh}$ satisfying $a\rho=\rho a$)
is contained in $C_b$?
\smallskip

A  stronger condition than (\ref{01}), namely that (\ref{01}) holds for all $x\in\matn{\bh}$
and all $n\in\bn$ (where $a$ and $b$ are replaced by the multiples $a^{(n)}$ and $b^{(n)}$ acting on
$\h^n$), implies that (\ref{01}) holds in any representation of the C$^*$-algebra generated by $a,b$ and $1$
(Lemma \ref{le61}) and that $b$ is contained in the C$^*$-algebra generated by $a$ and $1$ (Corollary \ref{co610}).
In a special situation (when $\h$ is a cogenerator for Hilbert modules over the operator algebra  
$A_0$ generated by $a$ and $1$) it follows that  $b$ must be in $A_0$ (Proposition \ref{th62}). 
If $a$ is, say, subnormal (a restriction of a normal operator to an invariant subspace), this means that $b=f(a)$ for 
a function $f$ in the uniform closure of polynomials
on $\sigma(a)$. Perhaps for a general subnormal operator $a$ (\ref{01}) does not imply
that $b=f(a)$ for a function $f$, but when it does, it forces on $f$ certain degree of
regularity. For example, if $a$ is the operator of multiplication on the Hardy space $H^2(G)$ by the 
identity function on $G$, where
$G$ is a domain in $\bc$ bounded by finitely many nonintersecting analytic Jordan curves, (\ref{01}) implies
that $b$ is an analytic Toeplitz operator with a symbol $f$ which is continuous also on the boundary of $G$ 
(Proposition \ref{pr6101}).

Let us  call a complex function
$f$ on a compact set $K\subseteq\bc$ a {\em Schur function} if the supremum over all (finite) sequences 
$\lambda=(\lambda_1,\lambda_2,\ldots)\subseteq K$ of  norms of  matrices
$$\Lambda(f;\lambda)=\left[\frac{f(\lambda_i)-f(\lambda_j)}{\lambda_i-\lambda_j}\right],$$ 
regarded as Schur multipliers, is finite. (Here the quotient is interpreted as $0$ if $\lambda_i=\lambda_j$.) 
If $a$ is normal the work of Johnson and Williams \cite{JW} tells us
that  $b=f(a)$ satisfies (\ref{01}) if and only if $f$ is a Schur function on $\sigma(a)$. 
In the `only if' direction we extend this to general subnormal operators (Proposition \ref{pr521}), in the other direction
only to subnormal operators with nice spectra (Theorem \ref{th621}).  

In the last section  we will investigate the condition (\ref{01}) in the case when $a$ is subnormal and $b=f(a)$ for
a function $f$. If $a$ is normal, a known effective method of studying such commutator estimates is based on double
operator integrals (see \cite{APPS} and the references there), which are defined via spectral projection 
valued measures. But, since invariant subspaces of a normal operator are not necessarily invariant under its spectral
projections, a different method is needed for subnormal operators. In Section 7 we will
`construct' for a given subnormal operator $a$ and suitable function $f$ on $\sigma(a)$ a completely bounded
map $T_{a,f}$ on $\bh$ such that $T_{a,f}$ commutes with the left and the right multiplication by $a$ and $aT_{a,f}(x)-T_{a,f}(x)a=f(a)x-xf(a)$ for all $x\in\bh$. For $b=f(a)$ this implies
(\ref{01}) and also the range inclusion $d_b(\bh)\subseteq d_a(\bh)$. By the above mentioned
result from \cite{JW} even if $a$ is normal the functions $f$ considered here must be Schur.  By \cite{JW} every
Schur function on $\sigma(a)$ is complex differentiable relative to $\sigma(a)$ at each nonisolated point of $\sigma(a)$ 
(thus holomorphic on the interior of $\sigma(a)$) and $f^{\prime}$ is bounded. The construction of $T_{a,f}$ applies
to the subclass that includes all functions for which $f^{\prime}$ is Lipschitz of order $\alpha>0$. Only if $\sigma(a)$
is sufficiently nice are we able to find $T_{a,f}$ for all Schur functions.

We will denote by $\normc{S}$ the norm closure and by $\weakc{S}$ the weak* closure of a subset $S$ in $\bh$.

\section{Variance of operators}

\begin{de}For a bounded operator $a$ on a Hilbert space $\h$ and a nonzero vector $\xi\in\h$ let 
$$\dxa=(\|a\xi\|^2\|\xi\|^2-|\inner{a\xi}{\xi}|^2)\|\xi\|^{-2}.$$
Thus, if $\xi$ is a unit vector and $\omega:x\mapsto\inner{x\xi}{\xi}$ is the corresponding vector state on $\bh$,
then 
$$\dxa=\omega(a^*a)-|\omega(a)|^2,$$
and this formula can be used to define the {\em variance} $D_{\omega}(a)$ of $a$ in any (not just vector) state $\omega$.
\end{de}

\begin{re}\label{re0} (i) It is clear from the definition that $\dxa$ is just the square of the distance of $a\xi$ to the set 
$\bc\xi$ of scalar multiples of $\xi$. Hence, if $\dxb\leq\dxa$ for all $\xi\in\h$, then in particular each 
eigenvector of $a$ is also an eigenvector for $b$. Consequently $\dxb=0$ for all unit vectors $\xi\in\h$ if and only if $b\in\bc1$.

(ii) $\dis{\xi}{\alpha a+\beta1}=|\alpha|^2\dxa$ for all $a,b\in\bh$ and $\alpha,\beta\in\bc$.

(iii) $\dis{\xi}{a^*}=\dxa$ for all $\xi\in\h$ if and only if $a$ is normal.
\end{re}

\begin{theorem}\label{th1}If operators $a,b\in\bh$ satisfy $\dxb=\dxa$ for all $\xi\in\h$, then there exist $\alpha,\beta\in\bc$ with $|\alpha|=1$ such that
$b=\alpha a+\beta$ or $a$ is normal and $b=\alpha a^*+\beta$.
\end{theorem}

\begin{proof}We assume that neither $a$ no $b$ is a scalar multiple of the identity, otherwise the proof is easy. If $b$ is of the form $b=\alpha a+\beta 1$ ($\alpha,\beta\in\bc$), then the hypothesis $\dxb=\dxa$ for all $\xi\in\h$, that is, $|\alpha|^2\dxa=\dxa$, clearly implies that $|\alpha|=1$. To deduce a similar conclusion in the case whenf $b=\alpha a^*+\beta1$,  replacing $b$ with $b-\beta1$, we may assume that $b=\alpha a^*$.  From the hypothesis we have that $|\alpha|^2\dis{\xi}{a^*}=\dxa$, which implies that $a$ and $a^*$ have the same eigenvectors, hence, if $\dim\h<\infty$, we can see inductively that $a$ is normal. To prove the same in general, we consider the distance $d(a,\bc1):=\inf_{\lambda\in\bc}\|a-\lambda1\|=d(a^*,\bc1)$. We may assume that this distance (and also $d(b,\bc1)=|\alpha|d(a^*,\bc1)$) is achieved at $\lambda=0$
(otherwise we just consider $a-\lambda1$ instead of $a$). Then by \cite[Theorem 2]{Sta} there exists a sequence of unit vectors $\xi_n\in\h$ such that $\lim_n\inner{a\xi_n}{\xi_n}=0$ and $\lim_n\|a\xi_n\|=\|a\|$. From $\dis{\xi_n}{a}=|\alpha|^2\dis{\xi_n}{a^*}$ it now follows
$$\|a\|^2=\lim_{n\to\infty}\|a\xi_n\|^2=|\alpha|^2\lim_{n\to\infty}\|a^*\xi_n\|^2\leq|\alpha|^2\|a^*\|^2=|\alpha|^2\|a\|^2.$$
This implies that $|\alpha|\geq1$ and similarly we prove (by exchanging the roles of $a$ and $a^*$) that $|\alpha|\leq1$. Thus $|\alpha|=1$ and from $\dis{\xi}{a^*}=\dxa$ (for all $\xi\in\h$) we now see that $a$ must be normal. To prove the theorem, we will now assume that neither $b$ nor $b^*$ is of the form $\alpha a+\beta1$ and show that this leads to a contradiction.

For any two nonzero vectors $\xi,\eta\in\h$ we expand the function 
$$f(z):=\dis{\xi+z\eta}{a}\|\xi+z\eta\|^2=\|a(\xi+z\eta)\|^2\|\xi+z\eta\|^2-|\inner{a(\xi+z\eta)}{\xi+z\eta}|^2$$
of the complex variable $z$ into powers of $z$ and $\overline{z}$,
$$f(z)=\dxa\|\xi\|^2+2\rea(D_1z)+2\rea(D_2z^2)+D_3|z|^2+2\rea(D_4|z|^2z)+\dis{\eta}{a}\|\eta\|^2|z|^4.$$
Among the coefficients $D_j$ we will need to know only $D_2$, which is
$$D_2=\inner{a\eta}{a\xi}\inner{\eta}{\xi}-\inner{a\eta}{\xi}\inner{\eta}{a\xi}.$$
Thus, from the equality $\dis{\xi+z\eta}{a}=\dis{\xi+z\eta}{b}$, by considering the coefficients of $z^2$
we obtain
$$\inner{b\eta}{b\xi}\inner{\eta}{\xi}-\inner{b\eta}{\xi}\inner{\eta}{b\xi}=\inner{a\eta}{a\xi}\inner{\eta}{\xi}-\inner{a\eta}{\xi}\inner{\eta}{a\xi}.$$
From this we see that if $\eta$ is orthogonal to $\xi$ and $a\xi$ then $\eta$ must be orthogonal to $b\xi$ or to $b^*\xi$. In other words,
if for a fixed $\xi$ we denote
$$\h_0(\xi)=\{\xi,a\xi\}^{\perp},\ \ \ \h_1(\xi)=\{\xi,a\xi,b\xi\}^{\perp},\ \ \ \h_2=\{\xi,a\xi,b^*\xi\}^{\perp},$$
then $\h_0(\xi)=\h_1(\xi)\cup\h_2(\xi)$. Since $\h_j(\xi)$ are vector spaces, this implies that $\h_1(\xi)=\h_0(\xi)$ or else
$\h_2(\xi)=\h_0(\xi)$. In the first case we have $b\xi\in\bc\xi+\bc a\xi$, while in the second case $b^*\xi\in\bc\xi+\bc a\xi$.
Since this holds for all $\xi\in\h$, it follows  that $\h$ is the union of the two sets
$$F_1=\{\xi\in\h:\, b\xi\in\bc\xi+\bc a\xi\}\ \ \ \mbox{and}\ \ \ F_2=\{\xi\in\h:\, b^*\xi\in\bc\xi+\bc a\xi\}.$$
Since $F_1$ and $F_2$ are closed, by Baire's theorem at least one of them has nonempty interior $\stackrel{\circ}{F_i}$. 
We will consider the case when $\stackrel{\circ}{F_1}\ne\emptyset$ and in appropriate places point out the differences with the other case, which is similar. Since $a\notin\bc1$, there exists a vector $\xi\in \stackrel{\circ}{F_1}$
such that $\xi$ and $a\xi$ are linearly independent. (Namely, if $a\xi=\alpha_{\xi}\xi$ for all $\xi\in\stackrel{\circ}{F_1}$,
where $\alpha_{\xi}\in\bc$, then considering this equality for the vectors $\xi$, $\zeta$ and $(1/2)(\xi+\zeta)$ in $\stackrel{\circ}{F_1}$,
where $\xi$ and $\zeta$ are linearly independent,
it follows easily that $\alpha_{\xi}$ must be independent of $\xi$ for $\xi$ in an open subset of $\h$, hence $a$ must be a scalar
multiple of $1$.) Let
$$U=\{\xi\in\stackrel{\circ}{F_1}:\, \xi\ \mbox{and}\ a\xi\ \mbox{are linearly independent}\}.$$
For any $\xi,\eta\in U$ and $z\in\bc$ let $\xi(z)=(1-z)\xi+z\eta$. If $\xi$ and $\eta$ are such that the `segment' $\xi(z)$  ($|z|\leq1$) is contained in $U$, then we have 
\begin{equation}\label{1}b\xi(z)=\alpha(z)a\xi(z)+\beta(z)\xi(z)\end{equation}
for some scalars $\alpha(z),\beta(z)\in\bc$. To see that the coefficients $\alpha$ and $\beta$ are holomorphic (in fact rational) functions of $z$ (for fixed $\xi$ and $\eta$), 
for any fixed $z_0$ with $|z_0|\leq1$ we take the inner product of both sides of (\ref{1})
with the vectors $\xi(z_0)$ and $a\xi(z_0)$ to obtain two equations from which we compute $\alpha(z)$ and $\beta(z)$ by Cramer's rule
(if $z$ is near $z_0$). (From the condition 
$\dis{\xi(z)}{b}=\dis{\xi(z)}{a}$ and (\ref{1}) we also conclude that $|\alpha(z)|=1$, hence $\alpha$ must be constant, which we could use to somewhat simplify the proof in the present case. But this argument is not available in the other case, when $\stackrel{\circ}{F_2}\ne\emptyset$, since we do not know if $\dis{\xi(z)}{b^*}=\dis{\xi(z)}{a}$, hence we will not use it.) Since $\alpha$ and $\beta$ are rational functions, it follows from (\ref{1}) that $b\xi(z)$ is contained in the two-dimensional space $S(z)$ spanned by $\xi(z)$ and $a\xi(z)$ for all $z\in\bc$. (Here we have used that the singular points are isolated and that the set of $z$ for which $b\xi(z)\notin S(z)$ is open.) It is known that this implies, since $b$ is not in $L:=\bc1+\bc a$, that $L$ contains an operator of rank one (see \cite[2.5]{MS} and use that $1\in L$). Thus, replacing $a$ by $a+\lambda1$ for a suitable $\lambda\in\bc$, we may assume that $a$ is of rank $1$. Let $\h_0$ be a $2$-dimensional subspace of $\h$ containing the range of $a$ and the orthogonal complement of the kernel of $a$ and set $\k=\ort{\h_0}$.  Then $b\xi\in\bc\xi$ for all $\xi\in\k$, which easily implies that $b|\k=\lambda1|\k$ for a scalar $\lambda$. Replacing $b$ by $b-\lambda1$, we may assume that $b|\k=0$. Since also $b\xi\in\bc a\xi+\bc\xi\subseteq\h_0$ for all $\xi\in\h_0$, we see that both operators $a$ and $b$ now live on the two dimensional space $\h_0$. (In the case when $\stackrel{\circ}{F_2}\ne\emptyset$, the same arguments reduce the proof to the case when $a$ and $b^*$, hence also $b$, live on the same two dimensional space.) Thus, it only remains to show that for
two $2\times2$ complex matrices $a$ and $b$ the condition $\dxb=\dxa$  implies that $b\in\bc a+\bc1$.
Now instead of proving this here directly we just refer to Lemma \ref{le2}(i) below, where a sharper result is proved.  
\end{proof}

\begin{co}\label{co11}If elements $a,b$ in a C$^*$-algebra $A\subseteq\bh$ satisfy $\dob=\doa$ for all pure states $\omega$ on $A$, then there is a
projection $p$ in the center $Z$ of the weak* closure $R$ of $A$ and central elements $u_1,z_1\in Rp$, $u_2,z_2\in R\ort{p}$, with $u_1,u_2$ unitary,
such that $bp=u_1a+z_1$ and $b\ort{p}=u_2a^*+z_2$ and $a\ort{p}$ is normal.
\end{co}

\begin{proof}Since the condition $\dob=\doa$ persists for all weak* limits of pure states on $A$ and such states are precisely the restrictions of weak*
limits of pure states on $R$ by \cite[Theorem 5]{G}, the proof immediately reduces to the case $A=R$. Let  $Z$ be the center of $R$, $\Delta$ the maximal ideal 
space of $Z$,  for each $t\in \Delta$ let $Rt$ be the closed ideal of $R$ generated
by $t$ and set $R(t):=R/(Rt)$. For any $a\in R$ let $a(t)$ denotes the coset of $a$ in $R(t)$.  Since each pure state on $R(t)$ can be lifted to a pure 
state on $R$, we have $D_{\omega}(b(t))=D_{\omega}(a(t))$ for each pure state $\omega$ on $R(t)$ and each $t\in\Delta$. Since $R(t)$ is a primitive C$^*$-algebra 
by \cite{Halp}, it follows from Theorem \ref{th1} that there exist scalars $\alpha(t),\beta(t)$, with $|\alpha(t)|=1$, such that 

\begin{equation}\label{11} b(t)=\alpha(t)a(t)+\beta(t)1\end{equation}
or
\begin{equation}\label{12} b(t)=\alpha(t)a(t)^*+\beta(t)1\ \mbox{and}\ a\ \mbox{is normal}.\end{equation}
Let $F_1$ be the set of all $t\in\Delta$ for which (\ref{11}) holds, $F_2$ the set of all those $t$ for which (\ref{12}) holds and $U$ the set of
all $t$ such that $a(t)$ is not a scalar. Since for each $x\in R$ the function $t\mapsto\|x(t)\|$ is continuous on $\Delta$ 
by \cite{G}, it is easy to see that $U$ is open and $F_1$, $F_2$ are closed. 

To show that the coefficients $\alpha$ and $\beta$
in (\ref{11}) and (\ref{12}) are continuous functions of $t$ on $U$, let $t\in U$ be fixed, note that the center
of $R(t)$ is $\bc1$ and that $R(t)$ is generated by projections, so there is a projection $p_t\in R(t)$ such that
$(1-p_t)a(t)p_t\ne0$. We may lift $1-p_t$ and $p_t$ to positive elements $x,y$ in $R$ with $xy=0$ \cite[4.6.20]{KR}. Then from 
(\ref{11}) 
\begin{equation}\label{111}x(s)b(s)y(s)=\alpha(s)x(s)a(s)y(s)\ (\forall s\in\Delta),\end{equation}
and $\|(xay)(s)\|\ne0$ for $s$ in a neighborhood
of $t$ by continuity. Now let $c\in Z$ be the element whose Gelfand transform is the function $s\mapsto\|x(s)a(s)y(s)\|$
and let $\phi:R\to Z$ be a bounded $Z$-module map such that $\phi(xay)=c$. (Such a map may be obtained simply as the
completely bounded $Z$-module extension to $R$ of the map $Z(xay)\to Z$, $z\mapsto z(xay)$, since $Z$ is injective
\cite{BLM}.) Since $\phi$ is a $Z$ module map, $\phi$ is just a collection of maps $\phi_s:R(s)\to Z(s)=\bc$, hence
from (\ref{111}) we obtain $\alpha(s)c(s)=(\phi(xby))(s)$. Since $c(t)=\|x(t)a(t)y(t)\|\ne0$, it follows that 
$\alpha$ is continuous in a neighborhood of $t$, hence continuous on $U$. Then, denoting by $q_0$ the projection corresponding to a clopen neighborhood
$U_0\subseteq U$ of $t$, we have from (\ref{11}) that $\beta(t)q_0(t)=e(t)$, where $e=(bq_0-\alpha aq_0)\in Rq_0$, hence 
$\beta|U_0$ represents a central element of $Rq_0$ and is therefore continuous.

Since $\Delta$ (hence also $\weakc{U}$) is a Stonean space and $\alpha$ (hence also $\beta$) are bounded continuous functions, they have
continuous extensions to $\weakc{U}$ (see \cite[p. 324]{KR}). If $q\in Z$ is the projection that corresponds to $\weakc{U}$, then $a\ort{q}$ is a 
scalar
in $R\ort{q}$, and it follows easily that $b\ort{q}$ must also be a scalar. So we have only to consider the situation in $Rq$, which means that
we may assume that $\weakc{U}=\Delta$, hence that $\alpha$ and $\beta$ are defined and continuous throughout $\Delta$. The interior $F:=\stackrel{\circ}{F}_1$
of $F_1$ is a clopen subset such that (\ref{11}) holds for $t\in F$. Since the complement $F^c\ (=\weakc{F_1^c}$) is contained in $F_2$, (\ref{12})
holds if $t\in F^c$. Finally, to conclude the proof, just let $p\in Z$ be the projection that corresponds to $F$, and let $u_1, z_1\in Zp$, $u_2,z_2\in Z\ort{p}$ be elements that corresponds to functions
$\alpha|F$, $\beta|F$, $\alpha|F^c$ and $\beta|F^c$ (respectively).
\end{proof}

We note that the converse of Corollary \ref{co11} also holds, the proof follows easily from the well-known fact 
\cite[p. 268]{KR} that if $\omega$ is a pure state on
a C$^*$-algebra $R$ then $\omega(xz)=\omega(x)\omega(a)$ for all $x\in R$ and all $z$ in the center of $R$.

\section{The inequality $\dis{\xi}{b}\leq\dxa$}

\begin{lemma}\label{le1}For any two operators $a,b\in\bh$ and any state $\omega$ on $\bh$ the following estimate holds: 
$$|D_{\omega}(b)-D_{\omega}(a)|\leq 2\|b-a\|(\|a\|+\|b\|).$$
\end{lemma}

\begin{proof}Since $|\omega(b^*b-a^*a)|\leq\|b^*b-a^*a\|=\|(b^*-a^*)b+a^*(b-a)\|\leq\|b-a\|(\|b\|+\|a\|)$ and $\left||\omega(b)|^2-|\omega(a)|^2\right|=
(|\omega(b)|+|\omega(a)|)\left| |\omega(b)|-|\omega(a)|\right|\leq(\|a\|+\|b\|)\|b-a\|$, we have
\begin{align*}|D_{\omega}(b)-D_{\omega}(a)|=|\omega(b^*b-a^*a)-(|\omega(b)|^2-|\omega(a)|^2)|\\
\leq|\omega(b^*b-a^*a)|+\left|(|\omega(b)|^2-|\omega(a)|^2)|\right|\leq2\|b-a\|(\|a\|+\|b\|).
\end{align*}
\end{proof}

\begin{lemma}\label{le2}Let $a,b\in{\mathbb M}_2(\bc)$ ($2\times 2$ complex matrices), $0<\varepsilon<1/2$, and let $\alpha_i$ and $\beta_i$ ($i=1,2$) 
be the eigenvalues of $a$ and $b$ (respectively).

(i) If $\dxb\leq\dxa$ for all unit vectors $\xi\in\bc^2$, then $b=\theta a+\tau$ for some scalars $\theta,\tau\in\bc$ with $|\theta|\leq1$.

(ii) If $\dxb\leq\dxa+\varepsilon^8$ for all unit vectors $\xi\in\bc^2$, then 
$$|\beta_2-\beta_1|\leq|\alpha_2-\alpha_1|+2\varepsilon(\|a\|+2\|b\|+1).$$
\end{lemma}

\begin{proof} (i) Since $\dxa=\dis{\xi}{a-\lambda1}$ for all $\lambda\in\bc$, we may assume that one of the eigenvalues of $a$ is $0$,
say $a\xi_2=0$ for a unit vector $\xi_2\in\bc^2$. Then from $0\leq\dis{\xi_2}{b}\leq\dis{\xi_2}{a}=0$ we see that $\xi_2$ is also an
eigenvector for $b$, hence (replacing $b$ by $b-\lambda1$ for a $\lambda\in\bc$) we may assume that $b\xi_2=0$. So, choosing a suitable
orthonormal basis $\{\xi_1,\xi_2\}$ of $\bc^2$, we may assume that $a$ and $b$ are of the form
$$a=\left[\begin{array}{ll}
\alpha_1&0\\
\gamma&0
\end{array}\right],\ \ \ \ b=\left[\begin{array}{ll}
\beta_1&0\\
\delta&0\end{array}\right].$$
Now we compute for any unit vector $\xi=(\lambda,\mu)\in\bc^2$ (using $|\lambda|^2+|\mu|^2=1$) that
\begin{align*}\dxa=\|a\xi\|^2-|\inner{a\xi}{\xi}|^2=(|\alpha_1|^2+|\gamma|^2)|\lambda|^2-\left|\alpha_1|\lambda|^2+\gamma\lambda\overline{\mu}\right|^2\\
=|\lambda|^2[|\alpha_1|^2|\mu|^2+|\gamma|^2|\lambda|^2-2\rea(\alpha_1\overline{\gamma}\overline{\lambda}\mu)].
\end{align*}
Using this and a similar expression for $\dxb$, the condition $\dxb\leq\dxa$ can be written as
\begin{equation}\label{20}(|\alpha_1|^2-|\beta_1|^2)|\mu|^2+(|\gamma|^2-|\delta|^2)|\lambda|^2-
2\rea((\alpha_1\overline{\gamma}-\beta_1\overline{\delta})
\overline{\lambda}\mu)\geq0,
\end{equation}
which means that the matrix
$$M=\left[\begin{array}{ll}
|\alpha_1|^2-|\beta_1|^2&\beta_1\overline{\delta}-\alpha_1\overline{\gamma}\\
\overline{\beta}_1\delta-\overline{\alpha}_1\gamma&|\gamma|^2-|\delta|^2\end{array}\right]$$
is nonnegative. This is equivalent to the conditions 
$$|\beta_1|\leq|\alpha_1|,\ \ |\delta|\leq|\gamma|\ \ \mbox{and}\ \ \det M\geq0.$$
Since $\det M=-|\alpha_1\delta-\beta_1\gamma|^2$, the condition $\det M\geq0$ means that $\alpha_1\delta=\beta_1\gamma$. If $\alpha_1\ne0$,
it follows that $b$ is of the form
$$b=\left[\begin{array}{ll}
\beta_1&0\\
\frac{\beta_1}{\alpha_1}\gamma&0
\end{array}\right]=\frac{\beta_1}{\alpha_1}a=\theta a, \ \mbox{where}\ \theta:=\frac{\beta_1}{\alpha_1},\ \mbox{hence}\ |\theta|\leq1.$$
If $\alpha_1=0$, then $\beta_1=0$ (since $|\beta_1|\leq|\alpha_1|$), hence again $b=\theta a$, where $\theta=\delta/\gamma$ if $\gamma\ne0 $.

(ii) As above, replacing $a$ and $b$ by $a-\lambda1$ and $b-\mu1$, where $\lambda$ and $\mu$ are eigenvalues of $a$ and 
$b$,  we may assume that
$a$ and $b$ are of the form
$$a=\left[\begin{array}{ll}
\alpha_1&0\\
\gamma&0
\end{array}\right],\ \ \ \ b=\left[\begin{array}{ll}
\beta_1&\delta_2\\
\delta_1&0\end{array}\right].$$
The norms of the new $a$ and $b$ are at most two times greater than the norms of original ones, 
which will be taken into account in the final estimate.
If $\xi=(0,1)$, then $\dxb=|\delta_2|^2$ and $\dxa=0$, hence the condition $\dxb\leq\dxa+\varepsilon^8$ shows that $|\delta_2|\leq\varepsilon^4$.
Thus, denoting by $b_0$ the matrix
$$b_0=\left[\begin{array}{ll}
\beta_1&0\\
\delta_1&0
\end{array}\right],$$
we have that $\|b-b_0\|\leq\varepsilon^4$, hence by Lemma \ref{le1}
$$|\dxb-\dis{\xi}{b_0}|\leq2\varepsilon^4(\|b_0\|+\|b\|)\leq4\|b\|\varepsilon^4\ \mbox{for all unit vectors}\ \xi\in\bc^2.$$
It follows that
$$\dis{\xi}{b_0}\leq\dxa+4\|b\|\varepsilon^4+\varepsilon^8\leq\dxa+\varepsilon^4(4\|b\|+1).$$
The same calculation that led to (\ref{20}) shows now that
$$|\lambda|^2[(|\alpha_1|^2-|\beta_1|^2)|\mu|^2+(|\gamma|^2-|\delta_1|^2)|\lambda|^2-
2\rea((\alpha_1\overline{\gamma}-\beta_1\overline{\delta}_1)
\overline{\lambda}\mu)]\geq-\varepsilon^4(4\|b\|+1)$$
for all $\lambda,\mu\in\bc$ with $|\lambda|^2+|\mu|^2=1$. We may choose the arguments of $\lambda$ and $\mu$ so that 
$(\alpha_1\overline{\gamma}-\beta_1\overline{\delta}_1)\lambda\overline{\mu}$ is positive, hence the above inequality implies that
$$t[(|\alpha_1|^2-|\beta_1|^2)(1-t)+(|\gamma|^2-|\delta_1|^2)t-2|\alpha_1\overline{\gamma}-\beta_1\overline{\delta}_1|\sqrt{t(1-t)}]
\geq-\varepsilon^4(4\|b\|+1)$$
for all $t\in[0,1]$. Setting $t=\varepsilon^2$, it follows (since $|\gamma|\leq\|a\|$ and $|\delta_1|\leq\|b\|$) that
$$(|\alpha_1|^2-|\beta_1|^2)(1-\varepsilon^2)+(\|a\|^2+\|b\|^2)\varepsilon^2\geq-\varepsilon^2(4\|b\|+1),$$
hence $|\alpha_1|^2-|\beta_1|^2\geq-\varepsilon^2(\|a\|^2+\|b\|^2+4\|b\|+1)$, so
$$|\beta_1|\leq|\alpha_1|+\varepsilon(\|a\|+2\|b\|+1).$$
Taking into account that $\alpha_2$ and $\beta_2$ were initially reduced to $0$ (by which the norms of $a$ and $b$ may have increased at most
by a factor $2$), this proves (ii).
\end{proof}

The approximate point spectrum of an operator $a$ will be denoted by $\sap{a}$.

\begin{de}\label{de1} If $a,b\in\bh$ are such that $\dxb\leq\dxa$ for all $\xi\in\h$, then we can define a function
$f:\sap{a}\to\sap{b}$ as follows. Given $\alpha\in\sap{a}$, let $(\xi_n)$ be a sequence of unit vectors
in $\h$ such that $\lim\|(a-\alpha1)\xi_n\|=0$. Then from the condition $\dis{\xi_n}{b}\leq\dis{\xi_n}{a}$
we conclude that $\lim\|(b-\lambda_n1)\xi_n\|=0$, where $\lambda_n=\inner{b\xi_n}{\xi_n}$. We will show that the sequence
$(\lambda_n)$ converges, so we define
$$f(\alpha)=\lim\lambda_n.$$
\end{de}

\begin{pr}\label{pr2}The function $f$ is well-defined and Lipschitz: $|f(\beta)-f(\alpha)|\leq|\beta-\alpha|$ for all
$\alpha,\beta\in\sap{a}$.
\end{pr}

\begin{proof}To slightly simplify the computation, we assume that $a$ and $b$ are contractions; for general $a$ and $b$ the proof is essentially the same. Given $\varepsilon>0$, choose unit vectors $\xi,\eta\in\h$ such that
$$\|(a-\alpha1)\xi\|<\varepsilon\ \ \ \mbox{and}\ \ \ \|(a-\beta1)\eta\|<\varepsilon.$$
Let $p$ be the projection onto the span of $\{\xi,\eta\}$ and let $c$ be the operator on $p\h$ defined by
$c\xi=\alpha\xi$ and $c\eta=\beta\eta$. Then
\begin{equation}\label{31}\|a|_{p\h}-c\|^2\leq\|(a-c)\xi\|^2+\|(a-c)\eta\|^2<2\varepsilon^2.\end{equation}
Let $\lambda=\inner{b\xi}{\xi}$, $\mu=\inner{b\eta}{\eta}$ and $d$ the operator on $p \h$ defined by
$d\xi=\lambda\xi$ and $d\eta=\mu\eta$. Then, using the conditions $\dxb\leq\dxa$  and $D_{\eta}(b)\leq D_{\eta}(a)$, 
we have
\begin{align}\label{32}\|b|_{p\h}-d\|^2\leq\|(b-\lambda1)\xi\|^2+\|(b-\mu1)\eta\|^2
\leq\|(a-\alpha1)\xi\|^2+\|(a-\beta1)\eta\|^2<2\varepsilon^2.\end{align}
Now by Lemma \ref{le1} and Remark \ref{re0}(i)
and since $\|d\|\leq\|b\|$, $\|c\|\leq\|a\|$ we infer from (\ref{31}) and (\ref{32}) that
$$\dis{\xi}{d}\leq\dxb+4\|d-b|p\h\|\|b\|<\dxb+4\varepsilon\sqrt{2}\ {\rm and}\ 
\dis{\xi}{c}>\dxa-4\varepsilon\sqrt{2},$$
hence (since $\dxb\leq\dxa$)
$$\dis{\xi}{d}\leq\dis{\xi}{c}+8\varepsilon\sqrt{2}\ \ \mbox{for all}\ \xi\in\h\ \mbox{with}\ \|\xi\|=1.$$
By Lemma \ref{le2} (ii) we now conclude that
\begin{equation}\label{33}|\mu-\lambda|\leq|\beta-\alpha|+\kappa\varepsilon^{\frac{1}{8}},\end{equation}
where $\kappa$ is a constant.

If $(\xi_n)$ and $(\eta_n)$ are two sequences of unit vectors in $\h$ such that $\lim\|(a-\alpha1)\xi_n\|=0$
and $\lim\|(a-\beta1)\eta_n\|=0$, we infer from (\ref{33}) (since $\varepsilon$ can be taken to tend to $0$
as $n\to\infty$) that
\begin{equation}\label{34}\limsup|\mu_n-\lambda_n|\leq|\beta-\alpha|.\end{equation}
Further, if $\beta=\alpha$ and we put in (\ref{33}) $\lambda_n=\inner{b\xi_n}{\xi_n}$ instead of $\lambda$ and 
$\lambda_m=\inner{b\xi_m}{\xi_m}$ instead of $\mu$,
we conclude that  $(\lambda_n)$ is a Cauchy sequence, hence it converges to a point $\lambda\in\bc$. From
$\lim\|(b-\lambda_n1)\xi_n\|=0$ it follows now that $\lim\|(b-\lambda1)\xi_n\|=0$, hence $\lambda\in\sap{b}$. Similarly
the sequence $(\mu_n)=(\inner{b\eta_n}{\eta_n})$ converges to some $\mu$ and (\ref{34}) implies that
$$|\mu-\lambda|\leq|\beta-\alpha|.$$
This shows that $f$ is a well-defined Lipschitz function.
\end{proof}

\begin{theorem}\label{th3}Let $a,b\in\bh$. If $a$ is normal, then there exists a constant
$\kappa$ such that $\dxb\leq\kappa\dxa$ for
all $\xi\in\h$ if and only if  $b=f(a)$ for a Lipschitz function $f$ on $\sigma(a)$.
In this case $\dob\leq\kappa\doa$ for all states $\omega$.
\end{theorem}

\begin{proof}Assume that $\dxb\leq\dxa$ for all $\xi\in\h$. We may assume that $a$ is not a scalar (otherwise
the proof is trivial). First consider the case when
$a$ can be represented by a diagonal matrix ${\rm diag}\,
(\alpha_j)$ in some orthonormal basis $(\xi_j)$ of $\h$. If $f:\sigma(a)\to\sap{b}$ is 
defined as in Definition \ref{de1}, then $b\xi_j=f(\alpha_j)\xi_j$ for all $j$, hence
$b=f(a)$. 

For a general normal $a$, first suppose that $\h$ is separable. Then by Voiculescu's version of the Weyl-von Neumann-Bergh theorem \cite{Vo1}, 
given $\varepsilon>0$, there exists a diagonal normal operator
$c={\rm diag}\,(\gamma_j)$ such that $\|a-c\|_2<\varepsilon$, where $\|\cdot\|_2$ denotes
the Hilbert-Schmidt norm. Let $(\xi_j)$ be an orthonormal basis of $\h$ consisting
of eigenvectors of $c$, so that $c\xi_j=\gamma_j\xi_j$. Since $\dis{\xi_j}{b}\leq
\dis{\xi_j}{a}$, by Remark \ref{re0}(i) there exist scalars $\beta_j\in\bc$ such that $\|(b-\beta_j1)\xi_j\|
\leq\|(a-\gamma_j1)\xi_j\|=\|(a-c)\xi_j\|$, hence
$$\sum_j\|(b-\beta_j1)\xi_j\|^2\leq\sum_j\|(a-c)\xi_j\|^2<\varepsilon^2.$$
In particular $\|b-d\|<\varepsilon$, where $d$ is the diagonal operator defined by
$d\xi_j=\beta_j\xi_j$. Since $d$ and $c$ commute, it follows that
$$\|bc-cb\|=\|(b-d)c-c(b-d)\|<2\varepsilon\|c\|\leq2\varepsilon(\|a\|+\varepsilon)\leq4\varepsilon\|a\|\ (\mbox{if}\ 
\varepsilon\leq\|a\|),$$
hence also
$$\|ba-ab\|=\|(bc-cb)+b(a-c)-(a-c)b\|\leq4\varepsilon(\|a\|+\|b\|).$$
Since this holds for all $\varepsilon>0$, it follows that $a$ and $b$ commute.
If $a$ has a cyclic vector this already implies that $b$ is in $(a)^{\prime\prime}$
hence a measurable function of $a$, but in general we need an additional argument
to prove this. Let $f:\sigma(a)\to\sap{b}$ be defined as in Definition \ref{de1}.
(Note that $\sigma(a)=\sap{a}$ since $a$ is normal.) Let $e(\cdot)$ be the projection
valued spectral measure of $a$, $\xi\in\h$ any separating vector for the von Neumann algebra $(a)^{\prime\prime}$
generated by $a$ and $\varepsilon>0$. 
If $\alpha$ is any point in
$\sigma(a)$, $U$ is any Borel subset of $\sigma(a)$ containing $\alpha$ and 
$\xi_U:=\|e(U)\xi\|^{-1}e(U)\xi$,
then $\|(a-\alpha1)\xi_U\|$ converges to $0$ as the diameter of $U$ shrinks to $0$.
For each $U$ let $\beta_U=\inner{b\xi_U}{\xi_U}$ so that $\|(b-\beta_U1)\xi_U\|\leq
\|(a-\alpha1)\xi_U\|$; then $f(\alpha)=\lim_{U\to\{\alpha\}}\beta_U$ by the definition 
of $f$. Thus, since by Proposition \ref{pr2} $f$ is a Lipschitz function, 
for each $\alpha\in\sigma(a)$ there is an open neighborhood $U_{\alpha}$
with the diameter at most $\varepsilon$ such that $|f(\alpha)-\beta_U|<\varepsilon$
for all Borel subsets $U\subseteq U_{\alpha}$ and $|f(\alpha_2)-f(\alpha_1)|<\varepsilon$
if $\alpha_1,\alpha_2\in U_{\alpha}$. By compactness we can cover $\sigma(a)$ with finitely
many such neighborhoods $U_{\alpha_i}$ and this covering then determines a partition
of $\sigma(a)$ into finitely many disjoint Borel sets $\Delta_j$ (say $j=1,\ldots,n$)
such that each $\Delta_j$ is contained in some $U_{\alpha_{i(j)}}$. Let $e_j=e(\Delta_j)$.
Now we can estimate, denoting $\beta_j=\beta_{\Delta_j}$,
\begin{align*}\|(b-f(a))e_j\xi\|\leq&\|(b-\beta_{j}1)e_j\xi\|+|(\beta_{j}-f(\alpha_{i(j)})|
\|e_j\xi\|+\|(f(\alpha_{i(j)})1-f(a))e_j\xi\|\\
\leq&\|(a-\alpha_j1)e_j\xi\|+|(\beta_{j}-f(\alpha_{i(j)})|
\|e_j\xi\|+\|(f(\alpha_{i(j)})1-f(a))e_j\xi\|\\
\leq&3\varepsilon\|e_j\xi\|.
\end{align*}
(Here we have used the spectral theorem to estimate the term $\|(f(\alpha_{i(j)})1-f(a))e_j\xi\|$ from above by
$\sup_{\alpha\in\Delta_j}|f(\alpha_{i(j)}-f(\alpha)|\|e_j\xi\|\leq\varepsilon\|e_j\xi\|$.)
Since $b$ commutes with $a$, hence also with all spectral projections of $a$, it follows that
\begin{align*}\|(b-f(a))\xi\|^2=&\|\sumjn e_j(b-f(a))e_j\xi\|^2
=\sumjn \|e_j(b-f(a))e_j\xi\|^2\\
\leq&
9\varepsilon^2\sumjn\|e_j\xi\|^2=9\varepsilon^2\|\xi\|^2.\end{align*}
Thus $\|(b-f(a))\xi\|\leq3\varepsilon\|\xi\|$ and, since this holds for all $\varepsilon>0$ and separating vectors of $(a)^{\prime\prime}$
are dense in $\h$, we conclude that $b=f(a)$.

If $\h$ is not necessarily separable, $\h$ can be decomposed into an orthogonal sum of separable subspaces
$\h_k$ that reduce both $a$ and $b$ and are such that $\sigma(a|\h_k)=\sigma(a)$. For each $k$ there exists
a Lipschitz function $f_k$ such that $b|\h_k=f(a|\h_k)$. Since for any two $k,j$ the space $\h_k\oplus\h_j$ is also separable,
there also exists a function $f$ such that $b|(\h_j\oplus\h_k)=f(a|(\h_j\oplus\h_k))$ and it follows easily that
$f_k=f=f_j$. Thus $b=f(a)$.

Conversely, if $b=f(a)$ for a function $f$ such that $$|f(\alpha_2)-f(\alpha_1)|\leq \kappa|\alpha_2-\alpha_1|$$
for all $\alpha_1,\alpha_2\in\sigma(a)$ and some constant $\kappa$, then for a fixed unit vector $\xi\in\h$
denote by $\mu$ the probability measure on Borel subsets of $\sigma(a)$ defined by $\mu(\cdot)=\inner{e(\cdot)\xi}{\xi}$.
Since $\dxa$ is just the square of the distance of $a\xi$ to $\bc\xi$ and similarly for $\dxb$, the estimate
\begin{align*}\|(f(a)-f(\alpha))1\xi\|^2=&\int_{\sigma(a)}|f(\lambda)-f(\alpha)|^2\, d\mu(\lambda)\\
\leq& \int_{\sigma(a)}\kappa|\lambda-\alpha|^2\, d\mu(\lambda)=\kappa\|(a-\alpha1)\xi\|^2\end{align*}
implies that $\dxb\leq\kappa\dxa$. 

Finally, since any state $\omega$ is in the weak*-closure of the set of all convex combinations of vector states and each
such combination can be represented as a vector state on ${\rm B}(\h^n)$ for some $n\in\bn$, the argument of the previous
paragraph (applied to $a^{(n)}$ and $b^{(n)}=f(a^{(n)})$ implies that $\dob\leq\doa$.
\end{proof}
A variant of the above Theorem \ref{th3} was proved in \cite{JW} and generalized to C$^*$-algebras in \cite{BMS}, but
both under the much stronger hypothesis that $\|[b,x]\|\leq\kappa\|[a,x]\|$ for all elements $x$, where $[a,x]$ denotes
the commutator $ax-xa$. (See Lemma \ref{le41} below for the explanation of the connection between the two conditions.)

The following Corollary was proved in \cite[5.2]{BMS} for prime C$^*$-algebras, but under a much stronger assumption
about the connection between $a$ and $b$ instead of the inequality $\dob\leq\doa$ for pure states $\omega$.

\begin{co}\label{co31}Let $A$ be a unital C$^*$-algebra, $a,b\in A$, $a$ normal. If $\dob\leq\doa$ for all states
$\omega$ on $A$, then $b=f(a)$ for a function $f$ on $\sigma(a)$ such that $|f(\mu)-f(\lambda)|\leq|\mu-\lambda|$ for
all $\lambda,\mu\in\sigma(a)$. If $A$ is prime,
it suffices to assume
the condition for pure states only.
\end{co}

\begin{proof}The first statement follows immediately from Theorem \ref{th3} since we may assume that $A\subseteq\bh$ for 
a Hilbert space $\h$ and each vector state on $\bh$ restricts to a state on $A$. For the second statement, we note
that the C$^*$-algebra generated by $a$ and $b$ is contained in a separable prime C$^*$-subalgebra $A_0$ of $A$
by \cite[3.1]{EZ} (an elementary proof of this is in \cite[3.2]{M}), and $A_0$ is primitive by \cite[p. 102]{Ped}, hence we may assume that $A_0$ is an irreducible C$^*$-subalgebra
of $\bh$. But then each vector state on $\bh$ restricts to a pure state on $A_0$, and each pure state on $A_0$ extends
to a pure state on $A$.
\end{proof}

\begin{co}\label{co32}Let $a,b\in\bh$ satisfy $\dxb\leq\dxa$ for all $\xi\in\h$. If $a$ is essentially normal,
then this implies that $\dot{b}=f(\dot{a})$ for a Lipschitz function $f$ on the essential spectrum of $a$, where
$\dot{a}$ denotes the coset of $a$ in the Calkin algebra.
\end{co}

\begin{proof}Any state $\omega$ on the Calkin algebra can be regarded as a  state on $\bh$ annihilating the
compact operators. By Glimm's theorem (see \cite[10.5.55]{KR} or \cite{G}) such a state $\omega$ is a weak* limit of vector 
states, hence $\dob\leq\doa$. The conclusion follows now from Corollary \ref{co31}.
\end{proof}

\begin{theorem}\label{th038}Let $A\subseteq\bh$ be a C$^*$-algebra  $a,b\in A$ and $a$ normal. Denote by $R$ the weak* closure of $A$ and by $Z$ the center of $R$.
Then the inequality $\dob\leq\doa$ holds for all pure states
$\omega$ on $A$ if and only if $b$ is in the norm closure of the set $S$ of all elements of the form $\sum_jp_jf_j(a)$ (finite sum), where $p_j$ are
orthogonal  projections in $Z$ with the sum $\sum_jp_j=1$ and $f_j$ are functions on $\sigma(a)$ such that $|f_j(\mu)-f_j(\lambda)|\leq
|\mu-\lambda|$ for all $\lambda,\mu\in\sigma(a)$.
\end{theorem}

\begin{proof}Note that $g(a(t))=g(a)(t)$ for each continuous
function $g$ on $\sigma(a)$. We will use the notation from the proof of Corollary \ref{co11}. Similarly as in that proof, the condition that $\dob\leq\doa$ for all pure states $\omega$ on $A$ implies the same condition 
for all pure states on $R(t)$ for all $t\in\Delta$ and  
it follows then
from Corollary \ref{co31} that for each $t$ there exists a Lipschitz function $f_t$ on $\sigma(a(t))$ with the Lipschitz constant $1$ such that
$b(t)=f_t(a(t))$. By Kirzbraun's theorem each $f_t$ can be extended to a Lipschitz function on $\sigma(a)$, denoted again by $f_t$, with the same Lipschitz 
constant $1$. Given $\varepsilon>0$, since $\Delta$ is extremely disconnected and  for each $x\in R$ the function $t\mapsto\|x(t)\|$ is continuous 
on $\Delta$ by \cite{G}, 
each $t\in\Delta$ has a clopen neighborhood $U_t$ such that $\|f_t(a)(s)-b(s)\|\leq\varepsilon$ for all $s\in U_t$. Let $(U_j)$ be a finite covering
of $\Delta$ by such neighborhoods $U_j:=U_{t_j}$ and for each $j$ let $p_j$ be the central projection in $R$ that corresponds to the clopen set $U_j$,
and set $f_j:=f_{t_j}$.
Then $$\|b-\sum_jp_jf_j(a)\|\leq\varepsilon.$$
Since this can be done for all $\varepsilon>0$, $b$ is in the closure of the set $S$ as stated in the theorem.

Conversely, suppose that for each $\varepsilon>0$ there exists an element $c\in R$ of the form $c=\sum_jp_jf_j(a)$, where $p_j\in Z$ are projections 
with the sum $1$ and $f_j$ are Lipschitz functions with the Lipschitz constant $1$, such that $\|b-c\|<\varepsilon$. Then for each pure state $\omega$ on 
$R$ and $x\in R$, $z\in Z$
the equality $\omega(zx)=\omega(z)\omega(x)$ holds \cite[4.3.14]{KR}). In particular $\omega|Z$ is multiplicative, hence $\omega(p_{j_0})=1$ for
one index $j_0$ and $\omega(p_j)=0$ if $j\ne j_0$. It follows now by a straightforward computation that $\dis{\omega}{c}=\dis{\omega}{f_{j_0}(a)}$,
which is at most $\doa$ by the same computation as in the last part of the proof of Theorem \ref{th3}. Now, since $\|b-c\|<\varepsilon$,  it follows from Lemma \ref{le1} (by letting $\varepsilon\to0$)
that $\dob\leq\doa$.
\end{proof}

\section{Is a derivation determined by the norms of its values?}

Given an operator $a\in\bh$, we will denote by $d_a$ the derivation on $\bh$ defined by 
$$d_a(x)=ax-xa.$$ For any vectors $\xi,\eta\in\h$ we denote by $\xieta$ the rank one operator on $\h$ defined
by $(\xieta)(\zeta)=\inner{\zeta}{\eta}\xi.$ The following lemma enables us to interpret the results of the previous
section in terms of derivations. 

\begin{lemma}\label{le41}For each unit vector $\xi\in\h$ and $a\in\bh$ we have the equality $$\|d_a(\xixi)\|^2=\max\{\dxa,\dxaa\}.$$
Thus, if $a$ is normal, then $\|d_a(\xixi)\|^2=\dxa$.\end{lemma}

\begin{proof}Denote $x=\xixi$. The square of the norm of $d_a(x)=a\xi\otimes\xi^*-\xi\otimes(a^*\xi)^*$ is equal to the
spectral radius of the operator $T:=d_a(x)^*d_a(x)$, which is the largest eigenvalue of the restriction of
$T$ to the span $\h_0$ of $\xi$ and $a^*\xi$. If $\xi$ and $a^*\xi$ are linearly independent, then the
matrix of $T|\h_0$ in the basis $\{\xi,a^*\xi\}$ can easily be computed to be
$$\left[\begin{array}{cc}
\dxa&\inner{\xi}{a\xi}(\|a\xi\|^2-\|a^*\xi\|^2)\\
0&\dxaa\end{array}\right].$$
Thus $\|d_a(x)\|^2=\max\{\dxa,\dxaa\}$. By continuity (considering perturbations of $a$) we see that this equality holds even if $\xi$ and $a^*\xi$ are linearly
dependent .
\end{proof}

\begin{theorem}\label{th42}If $a,b\in\bh$ are such that 
\begin{equation}\label{400}\|[b,x]\|=\|[a,x]\|\ \ \mbox{for all}\ x\in\bh,\end{equation} 
then either
$b=\sigma a+\lambda1$ for some scalars $\sigma,\lambda\in\bc$ with $|\sigma|=1$ or there exist a unitary $u$ and 
scalars $\alpha,\beta,\lambda,\mu$ in $\bc$ with $|\beta|=|\alpha|$ such that $a=\alpha u^*+\lambda1$ 
and $b=\beta u+\mu1$.
\end{theorem}

A variant of this theorem was proved in \cite[5.3, 5.4]{BMS} in general C$^*$-algebras, but under the additional assumption
that $a$ and $b$ are normal. The methods in \cite{BMS} are different from those we will use below. The author is not able to deduce Theorem \ref{th42} as a direct consequence of the previous results; for a proof we will need two additional lemmas. We denote by $a^{(n)}$ the direct sum of $n$ copies of an operator $a\in\bh$,
thus $a^{(n)}$ acts on $\h^n$. We will also use the usual notation  $[x,y]:=xy-yx$, so that $d_a(x)=[a,x]$.

\begin{re}\label{re42}We will need the following, perhaps well-known, general fact:  for any bounded
linear operators $S,T:X\to Y$ between Banach spaces the inequality 
\begin{equation}\label{200}\|Tx\|\leq\|Sx\|\ \ (x\in X)\end{equation}
implies $\|T^{\sharp\sharp}v\|
\leq\|S^{\sharp\sharp}v\|$ ($v\in X^{\sharp\sharp}$), where $T^{\sharp\sharp}$ denotes the second adjoint of $T$. This  follows from  \cite[1.1, 1.3]{JW},
but here is a slightly more direct proof. The inequality (\ref{200}) simply means that there is a contraction $Q$ from the
range of $S$ into the range of $T$ such that $T=QS$. But then $T^{\sharp\sharp}=Q^{\sharp\sharp}S^{\sharp\sharp}$,
which clearly implies the desired conclusion.
\end{re}

The content of the following lemma was observed already by Kissin and Shulman in the proof of \cite[3.3]{KS}. 

\begin{lemma}\label{le42}\cite{KS} Let $a,b\in\bh$ and suppose that 
\begin{equation}\label{41}\|[b,x]\|\leq\|[a,x]\|\end{equation} for all $x\in\kh$. If $a$ is normal, then $\|[b^{(n)},x]\|\leq\|[a^{(n)},x]\|$ for all
$x\in\matn{\bh}$ ($n\times n$ matrices with the entries in $\bh$) and all $n\in\bn$.
\end{lemma}

\begin{proof}Since $d_a=\bdu{(d_a|\kh)}$ (the second adjoint
in the Banach space sense), it follows from Remark \ref{re42} that (\ref{41}) holds for all $x\in\bh$. 

Suppose now that $a$ is normal and note that $\com{(a)}$ is a C$^*$-algebra by the 
Fuglede-Putnam theorem.
Since (\ref{41}) holds for all $x\in\bh$,  $b\in(a)^{\prime\prime}$.
Further, by (\ref{41}) the map $[a,x]\mapsto[b,x]$ is a contraction from $d_a(\bh)$ to $d_b(\bh)$. Clearly this map
is a homomorphism of $\com{(a)}$-bimodules, hence by \cite[2.1, 2.2, 2.3]{S} it is a complete contraction,
which is equivalent to the conclusion of the lemma.
\end{proof}

\begin{re}\label{re400}
We will use below the following well-known fact. Given $c_j,e_j\in\bh$,
 {\em an identity of the form $\sumjn c_jxe_j=0$, if it holds for all $x\in\bh$,
implies that all $c_j$ must be $0$ if the $e_j$ are linearly independent.
(See e. g. \cite[Theorem 5.1.7]{AM}).}
\end{re}

We refer to \cite{BLM} or \cite{Pa} for the definition of the injective envelope of an operator space used in the
following lemma.
\begin{lemma}\label{le43}Let $\crr=d_a(\bh)$ and let $\cs$ be the operator system
$$\cs=\left\{\left[\begin{array}{ll}
\lambda&y\\
z^*&\mu\end{array}\right]: \ \lambda,\mu\in\bc\; \ y,z\in\crr\right\}.$$
If $a$ does not satisfy any quadratic equation over $\bc$ then the C$^*$-algebra $C^*(\cs)$ generated by $\cs$ is irreducible and 
the injective envelope $I(\cs)$ of $\cs$ is $\mattwo{\bh}$.
\end{lemma}

\begin{proof}Since $\cs$ contains the diagonal $2\times 2$ matrices with scalar entries, each element of $\com{S}$ (the commutant of $\cs$) is a block diagonal
matrix, that is, of the form $c\oplus e$, where $c,e\in\bh$. To prove the irreducibility of $C^*(\cs)$ means to prove that
each selfadjoint such element $c\oplus e$ is a scalar multiple of $1$. Since $c\oplus e$ commutes with elements of $\cs$, we
have that $cy=ye$ for all $y\in\crr$. Setting $y=ax-xa$ in the last identity we obtain
\begin{equation}\label{401}cax-cxa-axe+xae=0\ \ \mbox{for all}\ x\in\bh.\end{equation}
Since in (\ref{401}) the left
coefficients $ca,-c,-a$ and $1$ are not all $0$, it follows that $1,a,e,ae$ are linearly dependent. Thus, if $1,a$ and $e$ are linearly
independent, then $ae=\alpha1+\beta a+\gamma e$ for some scalars $\alpha,\beta,\gamma\in\bc$. Using this, we may rearrange
(\ref{401}) into
\begin{equation}\label{402}(ca+\alpha1)x+(\beta1-c)xa+(\gamma1-a)xe=0.\end{equation}
If $1,a$ and $e$ were linearly independent, then (\ref{402}) would imply that $a=\gamma1$, but this would be in contradiction with
the assumption about $a$. Hence  $1, a$ and $e$ are linearly dependent, say $e=\alpha1+\beta a$ ($\alpha,\beta\in\bc$). Then
(\ref{401}) can be rewritten as
\begin{equation}\label{404}(c-\alpha1)ax+(\alpha1-\beta a-c)xa+\beta xa^2=0.\end{equation}
Since $1,a$ and $a^2$ are linearly independent by assumption, we infer from (\ref{404}) that $\beta=0$ and $c=\alpha1$. But then $e=\alpha1$ and $c\oplus e=\alpha(1\oplus1)$. This proves the irreducibility of $C^*(\cs)$.

Since $\cs$ contains nonzero compact operators, the identity map on $\cs$ has a unique completely positive extension to
$C^*(\cs)$ by the Arveson boundary theorem \cite{Ar1}, which implies that $C^*(\cs)\subseteq I(\cs)$. (Otherwise a projection
$\bh\to I(\cs)$ restricted to $C^*(\cs)$ would be a completely positive extension of $id_{\cs}$, different from $id_{C^*(\cs)}$.)
But since $C^*(\cs)$ is irreducible and contains nonzero compact operators, it follows that $C^*(\cs)\supseteq\mattwo{\kh}$,
hence $I(\cs)$ must contain the injective envelope $I(\mattwo{\kh})$, which is known to 
be $\mattwo{\bh}$ \cite{BLM}.
\end{proof}

\begin{proof}[Proof of Theorem \ref{th42}]If $a$ (or $b$) is a scalar multiple of $1$
the proof is easy, so we assume from now on that this is not the case. If $a$ satisfies
a quadratic equation of the form
$$a^2+\beta a+\gamma1=0\ \ (\beta,\gamma\in\bc),$$
then each element of  $(a)^{\prime\prime}$ is a polynomial in $a$ (this holds for
any algebraic operator $a$ by \cite{Tu}), hence in particular $b$ is a linear polynomial in
$a$, say $b=\sigma a+\lambda1$. Then the condition (\ref{400}) obviously implies that
$|\sigma|=1$. Hence we may assume that $a$ does not satisfy any quadratic equation over
$\bc$. By Lemma \ref{le41} the assumption (\ref{400}) 
implies that $\max\{\dxb,\dxbb\}=\max\{\dxa,\dxaa\}$ for all unit vectors $\xi\in\h$, hence for each non-zero $\xi\in\h$ at least one of the following four equalities hold:
\begin{equation}\label{04}\dxb=\dxa,\ \dxb=\dxaa,\ \dxbb=\dxa,\ \dxbb=\dxaa.\end{equation}
Since the functions of the form $\h\ni\xi\mapsto\|\xi\|^2\dxa$ are continuous, it follows that $\h$ is the union of four closed sets $F_i$, where $F_1=\{\xi\in\h;\, \|\xi\|^2\dxb=\|\xi\|^2\dxa\}$ and so on.
By Bair's theorem at least one of the sets $F_i$ has nonempty interior and then, since functions of the form $\xi\mapsto\|\xi\|^2\dxa$ are polynomial  (more precisely, for any fixed vectors $\xi,\eta$ the function $z\mapsto \|\xi+z\eta\|^2D_{\xi+z\eta}$ is a polynomial in $z$ and $\overline{z}$), at least one of the equalities (\ref{04}) must hold for all nonzero $\xi\in\h$. In each case it follows then by Theorem \ref{th1} that  
$b$ must have the form $b=\sigma a+\lambda1$ or $b=\sigma a^*+\lambda1$, where $|\sigma|=1$. Moreover, in the second case, which we assume from now on (otherwise the proof is already completed), we deduce now from (\ref{400}) that $\|[a^*,x]\|=\|[a,x]\|$ for all $x\in\bh$, hence (setting $x=a$) $a$ must be normal.
Replacing $b$ by
$\alpha b+\beta$ for suitable $\alpha,\beta\in\bc$ we may assume without loss of
generality that $b=a^*$.

Denote by $\crr_a$ and $\crr_b$ the ranges of the 
derivations $d_a$ and $d_b$
and by $\cs_a$ and $\cs_b$ the corresponding operator systems (as in Lemma \ref{le43}). 
Since $a$ is normal, by Lemma \ref{le42} the map 
$$\phi:\crr_a\to\crr_b,\ \ \phi([a,x]):=[b,x]\ (x\in\bh)$$
is completely contractive and the same holds for its inverse. Hence $\phi$ is
completely isometric and consequently the map
$$\Phi:\cs_a\to\cs_b,\ \ \Phi\left(\left[\begin{array}{cc}
\alpha&y\\
z^*&\beta\end{array}\right]\right):=\left[\begin{array}{cc}
\alpha&\phi(y)\\
\phi(z)^*&\beta\end{array}\right]$$
is completely positive with completely positive inverse, hence also completely isometric
(see \cite{Pa}). But then $\Phi$ extends to a complete isometry $\psi$ between the injective
envelopes $I(\cs_a)$ and $I(\cs_b)$ (since both $\Phi$ and $\Phi^{-1}$ extend to
complete contractions which must be each other's inverse by rigidity). 
Since $a$ (and $b=a^*$) does not satisfy any quadratic equation over $\bc$, these 
injective envelopes are both
$\mattwo{\bh}$ by Lemma \ref{le43}. Hence $\psi$ is a unital surjective complete
isometry of $\mattwo{\bh}={\rm B}(\h^2)$. Thus by \cite[4.5.13]{BLM} or \cite[Ex. 7.6.18]{KR}
(and since all automorphisms of ${\rm B}(\h^2)$ are inner) $\psi$ is necessarily of the form 
$$\psi(y)=w^*yw\ \ (y\in{\rm B}(\h^2)),$$
where $w\in{\rm B}(\h^2)$ is unitary. Since by definition $\psi$ fixes the projections of 
$\h^2$ on the two summands, $w$ must commute with these two projections (by the
multiplicative domain argument, see \cite[p. 38]{Pa}), consequently $w$ is of the form 
$w=u\oplus v$ for unitaries $u,v\in\bh$. It follows now from the definition of
$\psi$ that $\phi$  is  of the form 
$$\phi(y)=uyv\ \ (y\in\crr_a),$$
that is $\phi([a,x])=u[a,x]v$. Hence $u[a,x]v=[b,x]$ for all $x\in\bh$, which can
be rewritten as
\begin{equation}\label{405}uaxv-uxav-bx+xb=0\ \ (x\in\bh).\end{equation}
Thus by Remark \ref{re400} we see from (\ref{405}) that $v,av,1,$ and $b$ are linearly dependent. Hence, if $1$, $v$ and $b$
are linearly independent, then $av=\alpha1+\beta b+\gamma v$, where 
$\alpha,\beta\,\gamma\in\bc$, and (\ref{405}) can be rewritten as
$$(ua-\gamma u)xv-(\alpha u+b)x+(1-\beta u)xb=0.$$
But by Remark \ref{re400} this implies in particular that $ua-\gamma u=0$, hence
$a=\gamma 1$, a possibility which we have excluded in the first paragraph of this proof. So we may assume that
$1,v$ and $b$ are linearly dependent. If $v$ were a scalar, say $v=\delta$, then (\ref{405}) could be rewritten as $(\delta ua-b)x-\delta uxa+xb=0$, which would imply that $1$, $a$ and $b$ are linearly dependent, a possibility already taken care of in the beginning of the proof. Thus we may assume that $v$ is not a scalar. Hence $b=\alpha1+\beta v$
for suitable $\alpha,\beta\in\bc$. Since $v$ is unitary and $a=b^*$, this concludes the proof.
\end{proof}

To extend Theorem \ref{th42} to C$^*$-algebras we need a lemma.

\begin{lemma}\label{le508}Let $A\subseteq\bh$ be a C$^*$-algebra,  $J$ a closed ideal in $A$, and let $a,b\in A$ 
satisfy $\|[b,x]\|\leq\|[a,x]\|$ for all $x\in A$. Then the same inequality holds for all $x\in\weakc{A}$ and also
for all cosets $\dot{x}\in A/J$.
\end{lemma}

\begin{proof}The statement about the quotient was observed already in \cite[Proof of 5.4]{BMS} and follows from the 
existence of a quasicentral approximate unit $(e_k)$ in $J$
\cite{Ar}. Namely, the conditions $\|[a,e_k]\|,\ \|[b,e_k]\|\to 0$ (from the definition of the quasicentral approximate 
unit) and the 
well-known property that $\|\dot{y}\|=\lim_k\|y(1-e_k)\|$ ($y\in A$) imply that 
$$\|[\dot{b},\dot{x}]\|=\lim_k\|[b,x](1-e_k)\|=\lim_k\|[b,x(1-e_k)]\leq\lim_k\|[a,x(1-e_k)]\|=\|[\dot{a},\dot{x}]\|.$$

Let $\bdu{A}$ be the universal von Neumann envelope of $A$ (= bidual of $A$) and regard $A$ as a subalgebra in
$\bdu{A}$ in the usual way. Since $\bdu{d_a}$ is just the derivation induced by $a$ on $\bdu{A}$, it follows from
Remark \ref{re42} that the condition $\|[b,x]\|\leq\|[a,x]\|$ holds for all $x\in\bdu{A}$. Since $\weakc{A}$ is
a quotient of $\bdu{A}$, it follows from the previous paragraph (applied to $\bdu{A}$ instead of $A$) that the
condition holds also in $\weakc{A}$.
\end{proof}

\begin{co}If $A$ is a C$^*$-algebra and $a,b\in A$ are such that 
$\|[b,x]\|=\|[a,x]\|$ for all $x\in A,$ 
then there exist a  projection $p$ in the center $Z$ of $\weakc{A}$ and elements $s,d\in Zp$ with $s$ unitary, and
$u,v,c,g,h\in Z\ort{p}$ with $u,v$ unitary, such that $bp=sa+d$ and $a\ort{p}=cu^*+g$, $b\ort{p}=vcu+h$.
\end{co}

\begin{proof}If $A$ is primitive the corollary follows immediately from Theorem \ref{th42} and Lemma \ref{le508}
since $\weakc{A}=\bh$ if $A$ is irreducibly represented on $\h$. In general, Lemma \ref{le508} reduces the proof
to von Neumann algebras, where the arguments are similar as in the proof of Corollary \ref{co11}, so we will omit
the details.
\end{proof}

\begin{co}If $\|[b,x]\|\leq\|[a,x]\|$ for all $x\in A$ then $\max\{\dob,\dobb\}\leq\max\{\doa,\doaa\}$ for all pure states $\omega$ on $A$.
\end{co}

\begin{proof}If $\pi:A\to\bhp$ is the irreducible representation  obtained from $\omega$ by the GNS construction,
then $\weakc{\pi(A)}=\bhp$, hence the corollary follows from Lemmas \ref{le508} and \ref{le41}.
\end{proof}

\section{An inequality between norms of commutators}

In this section we study the inequality 
\begin{equation}\label{50}\|[b,x]\|\leq\kappa\|[a,x]\|\ \ (\forall x\in\bh),\end{equation}
where $a,b\in\bh$ are fixed and $\kappa$ is a constant. For a normal $a$ it is proved in \cite{JW} that (\ref{50})
holds (for some $\kappa$) if and only if 

\begin{equation}\label{51}d_b(\bh)\subseteq d_a(\bh).\end{equation}

That for normal $a$ (\ref{50}) implies (\ref{51}) can be easily proved as follows. We have seen in the
proof of Lemma \ref{le42} that for normal $a$ the condition (\ref{50}) is equivalent to the fact that the map
$$d_a(x)\mapsto d_b(x)\ \ (x\in\bh)$$
is a completely bounded homomorphism of $\com{(a)}$-bimodules $d_a(\bh)\to d_b(\bh)$. Then this map can be extended
to a completely bounded $\com{(a)}$-bimodule endomorphism $\phi$ of $\bh$ by the Wittstock theorem (see \cite[3.6.2]{BLM}), hence we have
$$d_b(x)=\phi(d_a(x))=d_a(\phi(x))\ \ (x\in\bh).$$

When studying the connection between (\ref{50}) and (\ref{51}), it is useful to have in mind a fact (recalled below as Lemma \ref{le51})  concerning 
operators in ${\rm B}(X,Y)$, the space of all bounded linear operators
from $X$ into $Y$, where $X$ and $Y$ are Banch spaces. Denote by $\du{X}$ the dual of $X$ and by $\du{T}$ the adjoint of
$T\in{\rm B}(X,Y)$. The following is well-known (see  \cite{JW}).

\begin{lemma}\label{le51}Given $S,T\in {\rm B}(X,Y)$, the inclusion $\du{T}(\du{Y})\subseteq\du{S}(\du{Y})$ holds if and 
only if there exists a constant
$\kappa$ such that
\begin{equation}\label{61}\|T\xi\|\leq\kappa\|S\xi\|\end{equation}
for all $\xi\in X$.
\end{lemma}

Since $d_a=-\du{(d_a|\trh)}$, where $\trh$ is the ideal in $\bh$ of trace class 
operators, the following is just a special case of Lemma \ref{le51}.

\begin{co}\label{co52}Let $a,b\in\bh$. 

(i) The inclusion $d_b(\bh)\subseteq d_a(\bh)$ holds if and only if there exists a constant $\kappa$ such that $\|d_b(t)\|_1\leq\kappa\|d_a(t)\|_1$
for all $t\in\trh$.

(ii) The inclusion $d_b(\trh)\subseteq d_a(\trh)$ is equivalent to the existence of a constant $\kappa$ such that $\|d_b(x)\|\leq\kappa
\|d_a(x)\|$ for all $x\in\kh$ or (equivalently, by Lemma \ref{le508}) for all $x\in\bh$.
\end{co}

If $a$ is not normal, then the range inclusion (\ref{51}) does not necessarily imply that $b\in(a)^{\prime\prime}$
\cite{Ho}, hence it does not imply (\ref{50}). But we will prove that conversely 
(\ref{50}) implies (\ref{51}), if $a$ satisfies certain conditions which are more general than
normality. 

\begin{pr}\label{le53}Denote $\crr_a:=d_a(\bh)$.  
If  $\weakc{\crr_a}+\com{(a)}=\bh$, then for each $b\in\bh$ the condition (\ref{50})
implies that $\crr_b\subseteq\crr_a$. Moreover, if $\weakc{\crr_a}=\bh$, then there exists
a weak* continuous $\com{(a)}$-bimodule map $\phi$ on $\bh$ such that $d_b=\phi d_a=d_a\phi$.
\end{pr}

\begin{proof} By (\ref{50})  the correspondence $d_a(x)\mapsto d_b(x)$
extends to a bounded map $\phi_0$ from  $\normc{\crr_a}$ 
into $\normc{\crr_b}$ such that $\phi_0d_a=d_b$. Note that $\phi_0(d_a(\kh))\subseteq 
d_b(\kh)$. Recall that for normed spaces $Y\subseteq Z$ the weak* closure $\weakc{Y}$ of $Y$ in $\bdu{Z}$ can be naturally identified with $\bdu{Y}$, hence in particular $d_a(\kh)^{\sharp\sharp}=\weakc{d_a(\kh)}$  inside $\kh^{\sharp\sharp}=\bh$. It follows that $\phi:=(\phi_0|d_a(\kh))^{\sharp\sharp}$ is the weak*
continuous extension of $\phi_0$ to $\weakc{d_a(\bh)}=\weakc{d_a(\kh)}$ satisfying $\phi d_a=d_b$.
Since $\phi_0$ is an $\com{(a)}$-bimodule map, so must be $\phi$ by continuity, hence
in particular 
$$d_b(x)=\phi(d_a(x))=d_a\phi(x)\ \ \mbox{for all}\ x\in\weakc{\crr_a}=\weakc{d_a(\bh)}$$
and consequently $d_b(\weakc{\crr_a})\subseteq \crr_a$. Finally, to conclude the proof, note that
the assumption 
$\weakc{\crr_a}+\com{(a)}=\bh$ implies that $\crr_b=d_b(\weakc{\crr_a})$, since from (\ref{50})
$\com{(a)}\subseteq\com{(b)}=\ker d_b$ so that $\crr_b=d_b(\weakc{\crr_a}+
\com{(a)})=d_b(\weakc{\crr_a})$.
\end{proof}

By duality the condition $\weakc{d_a(\bh)}=\bh$ means that the kernel of $d_a|\trh$ is $0$, that is,
$\com{(a)}\cap\trh=0$. There are many Hilbert space operators $a$ which do not commute even with any nonzero compact
operator. This is so for example, if $a$ is normal and has no eigenvalues. (Namely, $\com{(a)}$ is a C$^*$-algebra
and contains the spectral projection $p$ corresponding to any nonzero eigenvalue of each $h=h^*\in\com{(a)}$.
If $h$ is compact, then  $p$ is of finite rank, hence $ap$, and therefore also $a$, has eigenvalues.) For a general normal
$a\in\bh$ we can decompose $\h$ into the orthogonal sum $\h=\h_1\oplus\h_2$, where
$\h_1$ is the closed linear span of all eigenvectors of $a$ and $\h_2=\h_1^{\perp}$. Then $a$ also decomposes
as $a_1\oplus a_2$, where $\com{(a_2)}$ contains no nonzero compact operators, while $a_1$ is diagonal in an orthonormal basis. (A general subnormal operator, however, can commute with a nonzero trace class operator even if it is pure; an example is in  \cite[2.1]{Wi}.) 

\begin{co}\label{th54}Let $a\in\bh$ and suppose that $\h$ decomposes into the
orthogonal sum $\h_1\oplus\h_2$ of two subspaces which are invariant under $a$, so that
$a=a_1\oplus a_2$, where $a_i\in{\rm B}(\h_i)$. If $a_1$ is a diagonalizable normal operator, while
$\com{(a_2)}\cap\trht=0$ and $\sigma_p(a_2)\cap\sigma_p(a_1)=\emptyset$, $\sigma_p(a_2^*)\cap\sigma_p(a_1^*)=\emptyset$, where $\sigma_p(c)$ denotes the set of all eigenvalues of an operator $c$, then the condition $\|d_b(x)\|\leq\|d_a(x)\|$ ($\forall x\in\bh$) implies that 
 $d_b(\bh)\subseteq d_a(\bh)$.
\end{co}

\begin{proof}Since  $b\in(a)^{\prime\prime}$,  $\h_1$ and
$\h_2$ are invariant subspaces for $b$, so  $b$ also decomposes as $b=b_1\oplus b_2$,
where $b_i\in{\rm B}(\h_i)$. Relative to the same decomposition of $\h$ each $x\in\bh$
can be represented by a $2\times2$ operator matrix $x=[x_{i,j}]$ and
$$d_b(x)=\left[\begin{array}{ll}
b_1x_{1,1}-x_{1,1}b_1&b_1x_{1,2}-x_{1,2}b_2\\
b_2x_{2,1}-x_{2,1}b_1&b_2x_{2,2}-x_{2,2}b_2
\end{array}\right].$$
Thus it suffices to show that for each pair $(i,j)$ of indexes  and for each $x_{i,j}\in{\rm B}(\h_j,\h_i)$ 
the element $b_ix_{i,j}-x_{i,j}b_j$ is in the range of the map 
$d_{a_i,a_j}$ defined on ${\rm B}(\h_j,\h_i)$ by $d_{a_i,a_j}(y)=a_iy-ya_j$. In the case
$i=2=j$ this follows from Proposition \ref{le53} and in the case $i=1=j$ this is an elementary special case of a result from \cite{JW}. We will now consider the case $i=2$ and
$j=1$, the remaining case $i=1$ and $j=2$ is treated similarly.

From the norm inequality condition we have in particular that 
$$\|d_{b_2,b_1}(x)\|\leq\|d_{a_2,a_1}(x)\|\ \ (\forall x\in
{\rm B}(\h_1,\h_2)).$$
This implies that there exists a bounded $\com{(a_2)},\com{(a_1)}$-bimodule map 
$$\phi_0:\normc{d_{a_2,a_1}({\rm K}(\h_1,\h_2))}\to{\rm K}(\h_1,\h_2)$$
such that $\phi_0d_{a_2,a_1}=d_{b_2,b_1}$. (To prove that $\phi_0$ is indeed a bimodule map, we use  that $\com{(a_i)}\subset\com{(b_i)}$,
which follows from $\com{(a)}\subseteq\com{(b)}$.) As in the proof of Proposition \ref{le53} we now extend $\phi_0$ weak* continuously to the weak* closure $\weakc{\mathcal R}$ of the range $\mathcal{R}$ of $d_{a_2,a_1}$ and show that $d_{b_2,b_1}(\weakc{\mathcal R})\subseteq\mathcal{R}$.
Finally, let $(\xi_j)_{j\in\bj}$ be an orthonormal basis of $\h_1$ consisting of eigenvectors of $a_1$ and let $\alpha_j$ be the corresponding eigenvalues. If $y\in\ker d_{a_1,a_2}$, then for each $j\in\bj$ and $\eta\in\h_2$ we have
$\inner{a_2^*y^*\xi_j-\overline{\alpha}_jy^*\xi_j}{\eta}=\inner{\xi_j}{ya_2\eta}-\inner{y^*a_1^*\xi_j}{\eta}=-\inner{\xi_j}{d_{a_1,a_2}(y)\eta}=0$, which means (by the arbitrariness of $\eta$) that $y^*\xi_j$ is an eigenvector for $a_2^*$ with the eigenvalue $\overline{\alpha}_j$. Since  by assumption $\sigma_p(a_2^*)\cap\sigma_p(a_1^*)=\emptyset$  and the vectors $\xi_j$ span $\h_1$, we infer that $y=0$. Thus $\ker d_{a_1,a_2}=0$.  Consequently $\overline{\mathcal{R}}$ (which is just the annihilator in ${\rm B}(\h_1,\h_2)$ of $\ker(d_{a_1,a_2}|{\rm T}(\h_2,\h_1)$) is equal to ${\rm B}(\h_1,\h_2)$. Therefore $d_{b_2,b_1}({\rm B}(\h_1,\h_2)=d_{b_2,b_1}(\overline{\mathcal{R}})\subseteq\mathcal{R}$.     
\end{proof}
 
{\em Problem.} Does  Corollary \ref{th54} still hold if we omit the hypothesis about the disjointness of the point spectra?

Perhaps, in general, (\ref{50})  does not even imply that $d_b(\bh)\subseteq\normc{d_a(\bh)}$,
but no counterexample is  known to the author. 
Note, however, that (\ref{50}) implies that $\ker d_a|\trh\subseteq\ker d_b|\trh$, hence by duality $d_b(\bh)\subseteq \weakc{d_a(\bh)}$; 
in particular $d_b(\kh)\subseteq\normc{d_a(\kh)}$ since
the weak topology agrees on $\kh$ with the weak* topology inherited from $\bh$. More generally, we will see that the 
question, whether   (\ref{50}) implies the inclusion  $d_b(\bh)\subseteq\normc{d_a(\bh)}$, depends entirely
on what happens in the Calkin algebra. 

For a C$^*$-algebra $A$ and $a\in A$ note that a functional $\rho\in\du{A}$ 
annihilates $d_a(A)$ if and
only if $[a,\rho]=0$, where $[a,\rho]\in\du{A}$ is defined by $([a,\rho])(x)=\rho(xa-ax)$. In other words, the
annihilator in $\du{A}$ of $d_a(A)$ is just the {\em centralizer} $C_a$ of $a$ in $\du{A}$.

\begin{pr}\label{pr06}If $a,b\in\bh$ satisfy $\|[b,x]\|\leq\|[a,x]\|$ for all $x\in\bh$, then 
$\|[\dot{b},\dot{x}]\|\leq\|[\dot{a},\dot{x}]\|$ in the Calkin algebra $\ch$. If this latter inequality
implies that $C_{\dot{a}}\subseteq C_{\dot{b}}$, then $C_a\subseteq C_b$ also holds, hence $d_b(\bh)\subseteq 
\normc{d_a(\bh)}$.
\end{pr}

\begin{proof}The first statement follows from Lemma \ref{le508}.
To prove the rest of the proposition, first note that for any $a\in\bh$ and a functional $\rho\in C_a$ the normal part $\rho_n$ and the singular 
part $\rho_s$ are both in $C_a$. (Indeed, from $[a,\rho]=0$ we have $[a,\rho_n]=-[a,\rho_s]$, where the left side is normal
and the right side is singular, hence both are $0$.) Further, since $\rho_n$ is given by a trace class operator $t$,
$[a,t]=0$, hence the hypothesis of the proposition implies that $[b,t]=0$, so $\rho_n\in C_b$. Since singular functionals
annihilate $\kh$, they can be regarded as functionals on the Calkin algebra $\ch$. Thus, if the condition
$\|[\dot{b},\dot{x}]\|\leq\|[\dot{a},\dot{x}]\|$ ($\dot{x}\in\ch$) implies that $C_{\dot{a}}\subseteq C_{\dot{b}}$,
then we have $\rho_s\in C_{\dot{b}}$, which means just that $\rho_s\in C_b$ (since $\rho_s$ annihilates $\kh$). Now both $\rho_n$ and $\rho_s$ are in $C_b$, hence so must be their sum $\rho$. This proves that $C_a\subseteq C_b$. The
Hahn-Banach theorem then implies that $d_b(\bh)\subseteq\normc{d_a(\bh)}$.
\end{proof}

\section{Commutators and the completely bounded norm}

In this section we will study  stronger variants of the condition $\|[b,x]\|\leq\|[a,x]\|$ ($x\in\bh$) in the context of
completely bounded maps.

\begin{lemma}\label{le61}If $a,b\in\bh$ satisfy
\begin{equation}\label{6100}\|[b^{(n)},x]\|\leq\|[a^{(n)},x]\|\ \ \mbox{for all}\ x\in\matn{\bh}\ \mbox{and all}\ n\in\bn,\end{equation}
then 
\begin{equation}\label{62}\|[\pi(b),x]\|\leq\|[\pi(a),x]\|\ \mbox{for all}\ x\in\bhp\end{equation}
for every unital $*$-representation $\pi:A\to\bhp$ of the C$^*$-algebra $A$ generated by $1,a$ and $b$.
\end{lemma}

\begin{proof}First assume that $\h_{\pi}$ is separable. Let $J=\kh\cap A$, $\h_n=[\pi(J)\h_{\pi}]$, and let $\pi_n$ and $\pi_s$ be the representations of $A$ defined by $\pi_n(a)=\pi(a)|\h_n$ and 
$\pi_s(a)=\pi(a)|\ort{\h_n}$ ($a\in A$), so that $\pi=\pi_n\oplus\pi_s$. By basic theory of representations of C$^*$-algebras of compact
operators $\pi_n$ is a subrepresentation of a multiple $id^{(m)}$ of the identity representation. By Voiculescu's theorem
(\cite{Vo}, \cite{Ar}) the representation $\pi\oplus id$ is approximately unitarily equivalent to $\pi_n\oplus id$, hence $\pi\oplus id$
is approximately unitarily equivalent to a subrepresentation of $id^{(m+1)}$. It follows easily from (\ref{6100}) that 
(\ref{62}) holds for any
multiple of the identity representation in place of $\pi$, hence it must also hold for any subrepresentation $\rho$ of $id^{(m+1)}$ (to see this,
just take in (\ref{62}) for $x$ elements that live on the Hilbert space of $\rho$). But then it follows from the approximate equivalence
that the condition (\ref{62}) holds for $\pi\oplus id$ in place of $\pi$, hence also for $\pi$ itself.

In general, when $\h_{\pi}$ is not necessarily separable, $\h_{\pi}$ decomposes into an orthogonal sum $\oplus_{i\in\bi}\h_i$ 
of separable invariant subspaces for $\pi(A)$. For a fixed $x\in\bhp$ there exists a countable subset $\bj$ of $\bi$
such that the norm of the operator $[\pi(b),x]$ is the same as the norm of its compression to $\l:=\oplus_{i\in\bj}\h_i$.
Since $\l$ is separable, it follows from what we have already proved that $\|[\pi(b),x]\|\leq\|[\pi(a),x]\|$. 
\end{proof}

\begin{co}\label{co610}If $a,b\in\bh$ satisfy (\ref{6100}) then $b$ is contained in the C$^*$-algebra $B$ generated by
$a$ and $1$.
\end{co}

\begin{proof}Let $\pi$ be the universal representation of $A=C^*(a,b,1)$ and $\h_{\pi}$ its Hilbert space. It follows
from Lemma
\ref{le61} (that is, from (\ref{62})) that $\pi(b)\in(\pi(a))^{\prime\prime}$, hence also $\pi(b)\in\pi(B)^{\prime\prime}$. 
But $\pi(B)^{\prime\prime}=\weakc{\pi(B)}$, thus $\pi(b)\in\weakc{\pi(B)}\cap\pi(A)=\pi(B)$, where the last equality is  
by \cite[10.1.4]{KR}.
\end{proof}

A completely contractive Hilbert module $\h$ over an operator algebra $A$ (that is, a Hilbert space on which 
$A$ has a completely
contractive representation) is a {\em cogenerator}  if for each nonzero morphism $R:\k\to\l$  of Hilbert $A$-modules (that is, a  bounded $A$-module map) there exists a morphism $T:\l\to\h$ such that $TR\ne0$ \cite[3.2.7]{BLM}.
Here by an operator algebra we will
always mean a norm complete algebra of operators on a Hilbert space.

\begin{pr}\label{th62}If $a,b\in\bh$ satisfy  (\ref{6100}), where $\h$ is a cogenerator
for the operator algebra $A_0$ generated by $a$ and $1$, then  $b\in A_0$.
\end{pr}

\begin{proof}Let $\pi$ be the universal representation of the C$^*$-algebra $A$ generated by $1,a$ and $b$. 
Then $\h_{\pi}$ (the Hilbert space of $\pi$) is a cogenerator for $A_0$. (Indeed, let $R:\k\to\l$ be a nonzero morphism of Hilbert $A_0$-modules and denote by $\rho$ the completely contractive representation of $A_0$ on $\l$ through which the $A_0$-module structure has been introduced on $\l$. There exists a representation $\sigma$ of $A$ on a Hilbert space $\l_1\supseteq\l$ such that $\rho(a)=\sigma(a)|\l$ for all $a\in A_0$ \cite{Pa}, hence $\l$ is a Hilbert $A_0$-submodule of $\l_1$. Thus $R(\k)\subseteq\l_1$. Since $\pi$ is universal (thus $\l_1$ is contained in a multiple of $\h_{\pi}$), there exists a morphism $T_1:\l_1\to\h_{\pi}$ of Hilbert $A$-modules  such that $T_1(R(\k))\ne0$. Then  $T:=T_1|\l:\l\to\h_{\pi}$ is a morphism of Hilbert $A_0$ modules such that $TR\ne0$.)  Hence 
by the Blecher-Solel bicommutation theorem (see \cite[3.2.14]{BLM}) 
$\weakc{\pi(A_0)}=\pi(A_0)^{\prime\prime}$. 
From Lemma \ref{le61} $\pi(b)\in \pi(A_0)^{\prime\prime}$, hence $\pi(b)\in\weakc{\pi(A_0)}\cap\pi(A)=\pi(A_0)$
by \cite[10.1.4]{KR}).
\end{proof}

The author does not know if in Proposition \ref{th62} the assumption that $\h$ is a cogenerator  is dispensable. In particular the following problem is open.

\smallskip
{\em Problem.} If in (\ref{6100}) $a$ is subnormal, is then $b$ necessarily of the form $b=f(a)$ for
some function $f$? Is $b$ necessarily subnormal?

\section{Commutators of functions of subnormal operators}

By Theorem \ref{th3} and Lemma \ref{le41} for a normal
operator $a$ the condition (\ref{50}) implies that $b=f(a)$
for a Lipschitz function $f$. However, as observed in \cite{JW}, (\ref{50}) implies that $f$
must have additional properties. In this section we will study properties of a function $f$ that imply or are implied by an inequality  of the form
\begin{equation}\label{598}\|[f(a),x]\|\leq\kappa\|[a,x]\| \ \ \forall x\in\bh,\end{equation}
where $a$ is a subnormal operator. 

\subsection{Schur functions}

Let us begin with the case when $a$ is a diagonal normal operator. 
Then there exists an orthonormal basis of $\h$ consisting of eigenvectors of $a$; let $\lambda_i$ be the corresponding
eigenvalues. If we denote by $[x_{i,j}]$ the matrix of a general operator $x\in\bh$ with respect to this basis,
then the inequality (\ref{598}) assumes the form
\begin{equation}\label{53}\|[(f(\lambda_i)-f(\lambda_j))x_{i,j}]\|\leq\kappa\|[(\lambda_i-\lambda_j)x_{i,j}]\|.
\end{equation}
In the same way we can express the inequality in Corollary \ref{co52}(i). Since there exist contractive projections from $\bh$ and from $\trh$ onto  subsets of block diagonal matrices, it follows  that the condition (\ref{598}) and its analogue for the trace norm are
equivalent to the requirements that the matrix $\Lambda(f)$ with the entries
\begin{equation}\label{54}\Lambda_{i,j}(f)=\left\{\begin{array}{ll}\frac{f(\lambda_i)-f(\lambda_j)}{\lambda_i-\lambda_j},&\mbox{if}\ \lambda_i\ne \lambda_j\\
0,&\mbox{if}\ \lambda_i=\lambda_j\end{array}\right.\end{equation}
is a Schur multiplier on $\bh$ and $\trh$ (respectively). In one direction  the last statement  can be generalized to subnormal operators.
 
\begin{pr}\label{pr521}Let $a\in\bh$ be a subnormal operator  and let $f$ be a Lipschitz function on $\sigma(a)$.
If $a$ is not normal, assume  that $f$ is in the uniform closure of the set of  rational functions with poles outside
$\sigma(a)$, so that $b:=f(a)$ is defined. If $\|[b,x]\|\leq\kappa\|[a,x]\|$ for all $x\in\bh$, then for
each sequence $(\lambda_i)\subseteq\sigma(a)$ the matrix $\Lambda(f;\lambda)$ with the entries  defined by the right side of (\ref{54}) is a Schur multiplier with
the norm at most $2\kappa$. That is, $f$ is a Schur function on $\sigma(a)$ as defined in the Introduction. Similarly,
the condition $\|[b,x]\|_1\leq\kappa\|[a,x]\|_1$ for all $x\in\trh$ implies that $\Lambda(f;\lambda)$ is a Schur multiplier on 
$\trh$ with the norm at most $2\kappa$. 
\end{pr}

\begin{proof}First suppose that $(\lambda_i)_{i=1}^m$ is a finite subset of the boundary $\partial\sigma(a)$ of $\sigma(a)$, where the $\lambda_i$ are distinct.
Then each $\lambda_i$ is an approximate eigenvalue of $a$ \cite{Co}, hence there exists a sequence of unit vectors $\xi_{i,n}\in\h$
such that $\lim_n\|(a-\lambda_i1)\xi_{i,n}\|=0$. Since $a-\lambda_i1$ is hyponormal, $\|(a-\lambda_i1)^*\xi_{i,n}\|\leq\|(a-\lambda_i1)\xi_{i,n}\|$
and it follows that the sequence
$$(\lambda_i-\lambda_j)\inner{\xi_{i,n}}{\xi_{j,n}}=\inner{\lambda_i\xi_{i,n}}{\xi_{j,n}}-
\inner{\xi_{i,n}}{\overline{\lambda}_j
\xi_{j,n}}$$ converges to $\lim_n(\inner{a\xi_{i,n}}{\xi_{j,n}}-\inner{\xi_{i,n}}{a^*\xi_{j,n}})=0.$
Thus $\lim\inner{\xi_{i,n}}{\xi_{j,n}}=0$  if $i\ne j$,  so the set $\{\xi_{1,n},\ldots,\xi_{m,n}\}$ is approximately orthonormal
if $n$ is large. Therefore for each matrix $\alpha=[\alpha_{i,j}]\in\matm{\bc}$ the norm of the operator 
$x:=\sum_{i,j=1}^m\alpha_{i,j}\xi_{i,n}\otimes\xi_{j,n}^*$ is approximately equal to the usual operator norm of $\alpha$. Further, for large $n$
we have approximate equalities
$$d_a(x)=\sum_{i,j=1}^m\alpha_{i,j}(a\xi_{i,n}\otimes\xi_{j,n}^*-\xi_{i,n}\otimes(a^*\xi_{j,n})^*)\approx\sum_{i,j=1}^m\alpha_{i,j}(\lambda_i-\lambda_j)
\xi_{i,n}\otimes\xi_{j,n}^*$$
and
$$d_b(x)\approx\sum_{i,j=1}^m\alpha_{i,j}(f(\lambda_i)-f(\lambda_j))\xi_{i,n}\otimes\xi_{j,n}^*,$$
hence it follows from the assumption $\|d_b(x)\|\leq\kappa\|d_a(x)\|$ that
\begin{equation}\label{201}\|[(f(\lambda_i)-f(\lambda_j))\alpha_{i,j}]_{i,j=1}^m\|\leq\kappa\|[(\lambda_i-\lambda_j)\alpha_{i,j}]_{i,j=1}^m\|.\end{equation}
By continuity the estimate (\ref{201}) holds also when $\lambda_i$ are not necessarily distinct. This estimate means that for a finite collection $\lambda=(\lambda_i)_{i=1}^m$ of not necessarily distinct elements of $\partial\sigma(a)$ the matrix $\Lambda(f;\lambda)$ with
the entries 
\begin{equation}\label{541}\Lambda_{i,j}(f;\lambda)=\left\{\begin{array}{ll}\frac{f(\lambda_i)-f(\lambda_j)}{\lambda_i-\lambda_j},&\mbox{if}\ \lambda_i\ne \lambda_j\\
0,&\mbox{if}\ \lambda_i=\lambda_j\end{array}\right.\end{equation}
acts as a Schur multiplier with the norm
at most $\kappa$ on the subspace $E\subseteq\matm{\bc}$ of matrices of the form $[(\lambda_i-\lambda_j)\alpha_{i,j}]$. 
Note that  $E$ (which depends on $\lambda_1,\ldots,\lambda_m$) is just the set 
of all matrices with zero entries on those positions $(i,j)$ for which $\lambda_i=\lambda_j$. Let $D$ be the subspace  of corresponding block diagonal
matrices (that is, matrices in $\matm{\bc}$ with non-zero entries only on those positions $(i,j)$ for which $\lambda_i=\lambda_j$). Since the natural projection from $\matm{\bc}$ onto $D$ has  Schur norm $1$ (and $\Lambda(f,\lambda)(D)=0$), it follows that the norm of $\Lambda(f;\lambda)$ as a Schur multiplier on $\matm{\bc}$ is at most $2\kappa$; the same bound $2\kappa$ is valid for all $m$. (Now it already follows from the second half of the proof of \cite[4.1]{JW}, that for each (non-isolated) point $\zeta\in\partial\sigma(a)$ the limit $f^{\prime}(\zeta):=\lim_{z\in\partial\sigma(a), z\to\zeta}\frac{f(z)-f(\zeta)}{z-\zeta}$ exists. So we can redefine the matrix $\Lambda(f;\lambda)$ by setting $\Lambda_{i,j}=f^{\prime}(\lambda_i)$ if $\lambda_i=\lambda_j$ (with $f^{\prime}(\lambda_i)$ interpreted as $0$ if $\lambda_i$ is isolated). Then (\ref{201}) and the continuity imply that the new $\Lambda(f,\lambda)$ has  Schur norm at most $\kappa$. But it is not necessary to use this redefined $\Lambda(f;\lambda)$ in this proof.)    

If $a$ is normal, then the above argument applies to all points of $\sigma(a)$ (not just points in $\partial\sigma(a)$) since all are approximate eigenvalues, hence we assume from now on that $a$ is not normal. Then by hypothesis $f$ is a uniform limit of rational functions with poles outside $\sigma(a)$, hence holomorphic on the interior
$G$ of $\sigma(a)$. We can use the first line of (\ref{541}) to define $\Lambda_{i,j}(f;\lambda)$ also for all pairwise distinct $\lambda_1,\ldots,\lambda_m$ from $G$.  When  $\lambda_i=\lambda_j\in G$ we do define $\Lambda_{i,j}(f;\lambda)$ by setting $\Lambda_{i,j}(f;\lambda)=f^{\prime}(\lambda_j)$.  For fixed elements $\lambda_2,\ldots,\lambda_m$ of $\partial\sigma(a)$ consider the function 
$$g(\lambda_1):=\Lambda(f;\lambda_1,\lambda_2,\ldots,\lambda_m)$$
from $\sigma(a)\setminus\{\lambda_2,\ldots,\lambda_m\}$ into the Banach algebra $S_m=\matm{\bc}$ equipped with the Schur norm. This function is holomorphic on $G$ and  (since $f$ is Lipschitz) bounded (by $m^2\kappa$). We would like to prove  that $g$ is bounded on $G$ by the same bound ($2\kappa$) as on $\partial\sigma(a)$, but we do not know if $g$ can be extended continuously to the closure $\overline{G}$ of $G$. (Namely, discontinuities can appear at the possible boundary points  $\lambda_2,\ldots,\lambda_m$.) We may consider the scalar valued functions $g_{\omega}=\omega g$  for all  linear functionals $\omega$ on $S_m$ with $\|\omega\|=1$. If for a fixed $\omega$ we denote $M=\sup_{\zeta\in\partial G\setminus\{\lambda_2,\ldots,\lambda_m\}}\lim_{z\to\zeta, z\in G}|g_{\omega}(z)|=\sup_{\zeta\in\partial G\setminus\{\lambda_2,\ldots,\lambda_m\}}|g_{\omega}(\zeta)|$ and $h(z)=|g_{\omega}(z)|-M$, then $h$ is subharmonic on $G$ and it follows from the extended maximum principle \cite[3.6.9]{Ra} (and the fact that finite sets are polar \cite[p. 56]{Ra}, while $\partial G$ is not polar since $G$ is bounded) that $h(\zeta)\leq0$ for all $\zeta\in G$. Thus $|h(\zeta)|\leq M$ for all $\zeta\in G$ and (since $M\leq\sup_{\zeta\in\partial G}\|g(\zeta)\|\leq2\kappa$) we deduce that $\sup_{\lambda_1\in G}\|g(\lambda_1)\|\leq2\kappa$. 
Thus the Schur norm of $\Lambda(f;\lambda_1,\lambda_2,\ldots,\lambda_m)$ is at most
$2\kappa$ for all $\lambda_1\in\sigma(a)$ and $\lambda_2,\ldots,\lambda_m\in\partial\sigma(a)$. In the same way, by considering the function $\lambda_2\mapsto\Lambda(f;\lambda_1,\lambda_2,\ldots,\lambda_m)$  for fixed $\lambda_1\in \sigma(a)$ and $\lambda_3,\ldots\lambda_m\in\partial\sigma(a)$, we can now show that the Schur norm of $\Lambda(f;\lambda_1,\ldots,\lambda_m)$ is at most $2\kappa$
for all $\lambda_1,\lambda_2\in \sigma(a)$ and $\lambda_3,\ldots,\lambda_m\in\partial\sigma(a)$.   Proceeding successively, we see that this must hold for all $\lambda_i \in\sigma(a)$ and, since the same bound $2\kappa$ is valid for all choices of $\lambda_1,\ldots,\lambda_m$ and all $m$, this implies that $f$ is a Schur function on $\sigma(a)$.  This proves the case of $\bh$ and the proof for $\trh$ is similar.
\end{proof}

It is well-known  that a matrix is  a Schur multiplier on
$\trh$ if and only if its transpose is a Schur multiplier on $\bh$ and the two multipliers have the same norm.
For a rank one operator $x$ the operators $d_a(x)$ and $d_b(x)$ have rank at most two and on such operators the trace class
norm is equivalent to the usual operator norm. If $a$ is normal, we deduce now from Corollary \ref{co52}, 
Lemma \ref{le41} and Theorem \ref{th3} that each of the two range inclusions  $d_b(\bh)\subseteq d_a(\bh)$ and
$d_b(\trh)\subseteq d_a(\trh)$ implies that $b$ is of the form $b=f(a)$ for a Lipschitz function $f$ on $\sigma(a)$. Then  $f$ is a
Schur function by  Proposition \ref{pr521}.

For  normal operators the converse of Proposition \ref{pr521} holds. Namely, let $a$ be normal and $f$  a
Schur function  on $\sigma(a)$. Given $\varepsilon>0$, 
by the Weyl-von Neumann-Bergh theorem
\cite[Corollary 39.6]{Co2} there exists a diagonal operator $a_0$ such that $\sigma(a_0)\subseteq\sigma(a)$, 
$\|a-a_0\|<\varepsilon$ and (approximating $f$
by polynomials) $\|f(a)-f(a_0)\|<\varepsilon$. Then, by what we have already proved for diagonal operators (by the computation preceding Proposition \ref{pr521}), for each $x\in\bh$ with $\|x\|=1$ we have $\|[f(a_0),x]\|\leq\kappa\|[a_0,x]\|$ for a constant $\kappa$, hence 
$\|[f(a),x]\|\leq\|[f(a_0),x]\|+2\varepsilon\leq\kappa\|[a_0,x]\|+2\varepsilon\leq\kappa\|[a,x]\|+2\kappa\varepsilon
+2\varepsilon$. Since this holds for all $\varepsilon>0$, we infer that $\|[f(a),x]\|\leq\kappa\|[a,x]\|$.  Thus we may summarize the above discussion
in  the following theorem  proved
already by Johnson and Williams in \cite{JW} in a somewhat different way.

\begin{theorem}\label{JW}\cite{JW} If $a\in\bh$ is normal, then for any $b\in\bh$ the inclusion $d_b(\bh)\subseteq d_a(\bh)$
holds if and only if there exists a constant $\kappa$ such that $\|d_b(x)\|\leq\kappa\|d_a(x)\|$ for all $x\in\bh$ and
this is also equivalent to the condition that $b=f(a)$ for a Schur function $f$ on $\sigma(a)$.
\end{theorem}

By \cite[6.5]{KS}, if $a$ is normal, (\ref{51}) implies (\ref{50}) in any C$^*$-algebra $A$. 
The converse is  true only under additional assumptions about $A$ 
(for example, if $A$ is a von Neumann algebra), but 
since the proof would considerably lengthen the paper, we will not present it  here.

Following the usual convention, we denote by ${\rm Rat}(K)$  the algebra of all rational functions with poles outside a compact subset
$K\subseteq\bc$ and, if $\mu$ is a positive Borel measure on $K$, $R^2(K,\mu)$ is the closure in
$L^2(\mu)$ of ${\rm Rat}(K)$. As before, for $a\in\bh$ we denote by $\dot{a}$ the coset in the Calkin algebra $\ch$. The simplest example of an operator $a$ satisfying the conditions of our next proposition is the unilateral shift.

\begin{pr}\label{pr6101}Let $K$ be a compact subset of $\bc$, $a$ a subnormal operator with $\sigma(a)\subseteq K$ 
such that $a$ is cyclic for the algebra ${\rm Rat}(K)$ and  let
$c$ be the minimal normal extension of $a$. Assume that $\sigma(c)= \sigma(\dot{a})$, let $\mu$ be a scalar 
spectral measure
for $c$ such that $a$ is the multiplication on $\h:=R^2(K,\mu)$ by the identity function $z$. Denote by $p$ the 
orthogonal projection from $\k:=L^2(\mu)$ onto $\h$
and assume that the only function $h\in C(\sigma(c))+(L^{\infty}(\mu)\cap R^2(K,\mu))$ for which the operator $T_h$ defined by 
$T_h(\xi):=p(h\xi)$ ($\xi\in\h$)
is compact is $h=0$. Then for each $b\in\bh$ satisfying $\|[b,x]\|\leq\|[a,x]\|$ ($x\in\bh$)  there exists a function
$f\in C(\sigma(c))\cap R^2(K,\mu)$  such that $b=f(c)|\h$.

Moreover, if $K$ is the closure of a domain $G$ bounded by finitely many non-intersecting analytic Jordan curves and 
$a$ is the multiplication operator by $z$ on the 
Hardy space $H^2(G)$, $f$ can be
extended to a Schur function on $K$.
\end{pr}

\begin{proof}It is well-known that a rationally cyclic subnormal  operator $a$ can be
represented as the multiplication on $R^2(K,\mu)$ by the independent variable $z$ \cite[p. 51]{Co3} and that
$\com{(a)}=R^2(K,\mu)\cap L^{\infty}(\mu)$ by Yoshino's theorem \cite[p. 52]{Co3}. Since $b\in\com{(a)}$, 
it follows that $b$ is the multiplication on $R^2(K,\mu)$ by a function 
$f\in R^2(K,\mu)\cap L^{\infty}(\mu)$. Thus $b=T_f$ since $\h=R^2(K,\mu)$ is invariant under multiplications by functions from $R^2(K,\mu)\cap L^{\infty}(\mu)$.

It follows from Lemma \ref{le41} and Bair's theorem (as in the proof of Theorem \ref{th42}) that at least one of the inequalities  $\dxb\leq\dxa$, $\dxb\leq\dxaa$ holds for all nonzero $\xi\in\h$.
Since $a$ is essentially normal by the Berger-Shaw theorem \cite[p. 152]{Co3},  by Corollary \ref{co32} 
$\dot{b}=g(\dot{a})$ for a continuous
function $g$ on $\sigma(\dot{a})$. Further, since $c$ is normal and $a$ is subnormal and essentially normal, an easy computation
with $2\times 2$ operator matrices (relative to the decomposition $\k=\h\oplus\h^{\perp}$) shows that the operator
$\ort{p}c^*p$ is compact, hence (since also $\ort{p}cp=0$) $\dot{p}\dot{c}=\dot{c}\dot{p}$. Consequently the map $h\mapsto \dot{T}_h\ (=ph(c)|\h)$ from $C(\sigma(c))$ into the
Calkin algebra ${\rm C}(\h)$ is a $*$-homomorphism, thus it must coincide with the $*$-homomorphism $h\mapsto h(\dot{a})$ since
the two coincide on the generator ${\rm id}_{\sigma(c)}$. It follows in particular that $\dot{b}=g(\dot{a})=\dot{T_g}$, hence the operator
$T_{g-f}=T_g-b$ is compact. But by the hypothesis this is possible only if $g-f=0$, hence $f=g$, therefore continuous. 

In the case $a$ is the unilateral shift, $f$ is a continuous function on the circle and contained in the closure 
$P^2(\mu)$ of polynomials in $L^2(\mu)$, where $\mu$ is the normalized
Lebesgue measure on the circle. It is well known that such a function can be holomorphically extended to the disc $\bd$ such that the extension
(denoted again by $f$) is continuous on $\weakc{\bd}$. By Proposition \ref{pr521} $f$ is a Schur function on $\weakc{\bd}$.  
Similar arguments
apply to multiply connected domains bounded by analytic Jordan curves by \cite[2.11, 1.1]{Ab}, \cite[4.3, 9.4]{Mu}.
\end{proof}

\subsection{A sufficient degree of smoothness}

By Proposition \ref{pr521}  
the inequality (\ref{598}) can hold only for
Schur functions. But the author does not know if (\ref{598}) holds for all Schur functions and all subnormal operators $a$, we will prove this
for all Schur functions only if $\sigma(a)$ is nice enough (Theorem \ref{th621}). 

It 
follows from the proof in \cite[Theorem 4.1]{JW} that a Schur function $f$ is complex differentiable  in the sense that the limit
$f^{\prime}(\zeta_0)=\lim_{\zeta\to\zeta_0,\ \zeta\in\sigma(a)}(f(\zeta)-f(\zeta_0))/(\zeta-\zeta_0)$ exists 
at each non-isolated point 
of $\sigma(a)$. Moreover, from the Lipschitz condition on $f$ we see that
$f^{\prime}$ is bounded. However, the boundedness of $f^{\prime}$ is not sufficient for $f$ to be a Schur function.
When $a$ is selfadjoint it is proved in \cite[5.1]{JW} that (\ref{598}) holds if $f^{(3)}$ is continuous. We will prove (\ref{598})
for subnormal operators $a$ under a much milder condition on $f$ (for example, $f^{\prime}$ Lipschitz suffices), but perhaps  our condition on $f$ is still more restrictive than Peller's condition that $f$ is a restriction of a function from the 
appropriate Besov space (see \cite{Pe} and \cite{APPS}), which is sufficient when $a$ is normal.

We will start from the special case of the Cauchy-Green formula
\begin{equation}\label{599}g(\lambda)=-\frac{1}{\pi}\int_{\bc}\frac{\dbar g(\zeta)}{\zeta-\lambda}\, dm(\zeta),
\end{equation}
which holds for a compactly supported differentiable function $g$ such that $\dbar g$ is bounded. Here $m$ denotes the planar Lebesgue measure and 
$\dbar g=(1/2)(\frac{\partial g}{\partial x}+i\frac{\partial g}{\partial y})$. (The proof in \cite[20.3]{Ru}
is valid for functions with the properties just stated.)
We note that an operator calculus based on the Cauchy-Green formula was already developed by Dynkin \cite{Dy}, however we will 
need rather different results, specific to subnormal operators.

\begin{lemma}\label{le610}If $a\in\bh$ is a subnormal operator and $g:\bc\to\bc$ is a differentiable function with compact 
support such that $\dbar g$ is bounded and $\dbar g|\sigma(a)=0$, then
\begin{equation}\label{610}\inner{g(a)\eta}{\xi}=-\frac{1}{\pi}\int_{\bc\setminus\sigma(a)}\dbar g(\zeta)
\inner{(\zeta1-a)^{-1}\eta}{\xi}\, dm(\zeta),\ \ (\xi,\eta\in\h).
\end{equation}
\end{lemma}

\begin{proof}Let $c\in\bk$ be the minimal normal extension of $a$, $e$ the projection valued spectral measure of $c$ (which is $0$ outside $\sigma(c)\subseteq\sigma(a)$), 
$K=\sigma(a)$ and $H=\{\zeta\in\bc:\, \dbar g(\zeta)\ne0\}$. For fixed $\eta\in\h$ and $\xi\in\k$ denote by $\mu$ the measure $\inner{e(\cdot)\eta}
{\xi}$. Then by the spectral theorem $g(c)=\int_Kg(\lambda)\, de(\lambda)$ and $(\zeta1-c)^{-1}=\int_K(\zeta-\lambda)^{-1}\,de(\lambda)$
for each $\zeta\in\bc\setminus K$ (in particular for $\zeta\in H$ since $H\cap K=0$ because of $\dbar g|K=0$), hence by
(\ref{599})
\begin{align*}\inner{g(c)\eta}{\xi}=\int_Kg(\lambda)\, d\mu(\lambda)=-\frac{1}{\pi}\int_K\int_H
\dbar g(\zeta)(\zeta-\lambda)^{-1}\, dm(\zeta)\, 
d\mu(\lambda)\\=-\frac{1}{\pi}\int_H\dbar g(\zeta)\int_K
(\zeta-\lambda)^{-1}\, d\mu(\lambda)\, dm(\zeta)\\=-\frac{1}{\pi}\int_H
\dbar g(\zeta)\inner{(\zeta1-c)^{-1}\eta}{\xi}\,dm(\zeta)=-\frac{1}{\pi}\int_{\bc\setminus\sigma(a)}\dbar g(\zeta)
\inner{(\zeta1-a)^{-1}\eta}{\xi}\,dm(\zeta).\end{align*}
For all $\xi\in\ort{\h}$ the last integrand is $0$ since $(\zeta1-a)^{-1}\eta\in\h$, hence $g(c)\eta\in\h$. Thus $\h$ is an invariant subspace for
$g(c)$ and the usual definition of $g(a)$, namely $g(a):=g(c)|\h$ (see \cite[p. 85]{Co3}),  is compatible with
(\ref{610}). To justify the interchange of order of integration in the above computation, let $M=\sup_{\zeta\in\bc}|\dbar g(\zeta)|$
and let $R$ be a constant larger than the diameter of the set $H-K$, so that for each $\lambda\in K$ the disc $D(\lambda,R)$ with the center $\lambda$ and radius $R$ 
contains $H$. Introduce the polar coordinates
by $\zeta=\lambda+re^{i\phi}$. Then by the Fubini-Tonelli theorem
\begin{align*}\int_H\int_K|\dbar g(\zeta)||\zeta-\lambda|^{-1}|\,d|\mu|(\lambda)\, dm(\zeta)
\leq M\int_H\int_K|\zeta-\lambda|^{-1}\, d|\mu|(\lambda)\, dm(\zeta)\\=
M\int_K\int_H|\zeta-\lambda|^{-1}dm\,(\zeta)\, d|\mu|(\lambda)
\leq M\int_K\int_{D(\lambda,R)}|\zeta-\lambda|^{-1}\,dm(\zeta)\, d|\mu|(\lambda)\\=M\int_K\int_0^{2\pi}\int_0^R\, dr\,d\phi\,
d|\mu|=2\pi MR|\mu|(K)<\infty.
\end{align*}
\end{proof}

Now, if $a$ and $g$ are as in Lemma \ref{le610} and if $b=g(a)$, we may compute formally for each $x\in\bh$
\begin{align*}[b,x]=[g(a),x]=-\frac{1}{\pi}\int_{\bc\setminus\sigma(a)}\dbar g(\zeta)[(\zeta1-a)^{-1},x]\, dm(\zeta)\\=
-\frac{1}{\pi}\int_{\bc\setminus\sigma(a)}\dbar g(\zeta)(\zeta1-a)^{-1}[a,x](\zeta1-a)^{-1}\, dm(\zeta)=
[a,T_{a,g}(x)],\end{align*}
where
\begin{equation}\label{611}T_{a,g}(x):=-\frac{1}{\pi}\int_{\bc\setminus\sigma(a)}\dbar g(\zeta)(\zeta1-a)^{-1}x(\zeta1-a)^{-1}\, dm(\zeta).\end{equation}
The problem here is, of course, the existence of the integral in (\ref{611}). We have to show that the map
\begin{equation}\label{612}(\eta,\xi)\mapsto
-\frac{1}{\pi}\int_{\bc\setminus\sigma(a)}\dbar g(\zeta)\inner{x(\zeta1-a)^{-1}\eta}{(\overline{\zeta}1-a^*)^{-1}\xi}\, dm(\zeta)\end{equation}
is a bounded sesquilinear form on $\h$. The following 
lemma will be helpful. 

\begin{lemma}\label{le6111}Let $a$, $g$ and $K:=\sigma(a)$ be as in Lemma \ref{le610}. If 
\begin{equation}\label{6111}\kappa:=\sup_{\lambda\in K}\int_{\bc\setminus K}|\dbar g||\zeta-\lambda|^{-2}\,dm(\zeta)
<\infty,\end{equation}
then the sesquilinear form defined by (\ref{612}) is bounded by $\frac{2}{\pi}\|x\|\kappa$.
\end{lemma}

\begin{proof}For any $t>0$, using first the Schwarz inequality and then the inequality $\alpha\beta\leq\frac{1}{2}(t^2\alpha^2+t^{-2}\beta^2)$ 
($\alpha,\beta\geq0$) to estimate the inner product in the integral in 
(\ref{612}), we see that the integral in (\ref{612}) is dominated by
\begin{align*}
\|x\|\int_{K^c}|\dbar g(\zeta)|\|(\zeta1-a)^{-1}\eta\|\|(\overline{\zeta}1-a^*)^{-1}\xi\|\, dm(\zeta)\leq\\
\|x\|\frac{1}{2}[t^2\int_{K^c}|\dbar g(\zeta)|
\|(\zeta1-a)^{-1}\eta\|^2\, dm(\zeta)+t^{-2}\int_{K^c}|\dbar g(\zeta)|\|(\zeta1-a)^{-1}\xi\|^2\, dm(\zeta)].
\end{align*}
Using the notation 
from the proof of Lemma \ref{le610} (with $\mu(\cdot):=\inner{e(\cdot)\xi}{\xi}$) and (\ref{6111}), we have
\begin{eqnarray*}\int_{\bc\setminus K}|\dbar g(\zeta)|\|(\zeta1-a)^{-1}\xi\|^2\, dm(\zeta)
=\int_H|\dbar g(\zeta)|\int_K|\zeta-\lambda|^{-2}\, d\mu(\lambda)\, dm(\zeta)\\
=\int_K\int_H|\dbar g(\zeta)||\zeta-\lambda|^{-2}\,dm(\zeta)\,d\mu(\lambda)\leq\kappa\mu(K)=\kappa\|\xi\|^2.
\end{eqnarray*}
Since a similar estimate holds with $\eta$ in place of $\xi$, it follows that
$$\int_H|\dbar g(\zeta)|\|(\zeta1-a)^{-1}\eta\|\|(\overline{\zeta}1-a^*)^{-1}\xi\|\, dm(\zeta)\leq
\kappa(t^2\|\eta\|^2+t^{-2}\|\xi\|^2).$$
Taking the infimum over all $t>0$ we get 
$$\frac{1}{\pi}\int_H|\dbar g(\zeta)|\|t(\zeta1-a)^{-1}\eta\|\|t^{-1}(\overline{\zeta}1-a^*)^{-1}\xi\|\, 
dm(\zeta)\leq\frac{2}{\pi}\kappa\|\eta\|\|\xi\|.$$
\end{proof}

\begin{re}\label{61111}Lemma \ref{le6111} applies, for example, 
if $\dbar g$ is a Lipschitz function of order $\alpha$, that is 
$|\dbar g(\zeta)-\dbar g(\zeta_0)|\leq\beta|\zeta-\zeta_0|^{\alpha}$ ($\zeta,\zeta_0\in\bc$)
for some positive constants $\alpha$ and $\beta$, with $\dbar g|K=0$. 
In this case the integral (\ref{6111}) may be estimated by noting that the Lipschitz condition (together with $\dbar g|K=0$) 
implies that $|\dbar g(\zeta)|\leq
\beta\delta(\zeta,K)^{\alpha}$, where $\delta(\zeta,K)$ is the distance from $\zeta$ to $K$. Let $R>0$ be so large that
for each $\lambda\in K$ the closed dics $D(\lambda,R)$ with the center $\lambda$ and radius $R$ contains $H$, where $H$ is as in the proof of Lemma \ref{le610}. Introducing the
polar coordinates by $\zeta=\lambda+re^{i\phi}$, for each $\lambda\in K$ we have
\begin{align*}\int_{\bc\setminus K}|\dbar g(\zeta)||\zeta-\lambda|^{-2}\,dm(\zeta)
\leq\beta\int_H\frac{\delta(\zeta,K)^{\alpha}}{|\zeta-\lambda|^2}\, dm(\zeta)
\leq\beta\int_{D(\lambda,R)}|\zeta-\lambda|^{\alpha-2}\, dm(\zeta)\\
=2\pi\beta\alpha^{-1} R^{\alpha}.\end{align*}
\end{re}

\begin{de}A function $f$ on a compact subset $K\subseteq\bc$ is in the {\em class ${\rm L}(1+\alpha,K)$} (where $\alpha\in(0,1]$) if the limit 
\begin{equation}\label{63}f^{\prime}(\zeta_0)=\lim_{\zeta\to\zeta_0,\ \zeta\in\sigma(a)}\frac{f(\zeta)-f(\zeta_0)}{\zeta-\zeta_0}\end{equation}
exists for each (nonisolated) $\zeta_0\in K$ and if  there exists a constant $\kappa>0$ such that
\begin{equation}\label{64}|f(\zeta)-f(\zeta_0)-f^{\prime}(\zeta_0)(\zeta-\zeta_0)|\leq\kappa|\zeta-\zeta_0|^{1+\alpha}\end{equation}
and
\begin{equation}\label{65}|f^{\prime}(\zeta)-f^{\prime}(\zeta_0)|\leq\kappa|\zeta-\zeta_0|^{\alpha}\end{equation}
for all $\zeta,\zeta_0\in K$.
\end{de}

We need the following consequence of the Whitney extension theorem.

\begin{lemma}\label{le64}Each $f\in{\rm L}(1+\alpha,K)$ can be extended to a continuously differentiable function $g$ on $\bc$
with compact support such that $\dbar g$ 
is a Lipschitz function of order $\alpha$ and $\dbar g(\zeta)=0$ if $\zeta\in K$ (even though $K$ may have empty interior).
\end{lemma}

\begin{proof}It suffices to extend $f$ to a differentiable function $g$ with  $\dbar g$  and $\partial g$ Lipschitz of order $\alpha$ and $\dbar g|K=0$, for then we simply replace $g$ by $hg$, where $h$ is a smooth function (that is, has continuous partial derivatives of all orders)  
with  compact support which is equal to $1$ on $K$. (Namely, $gh$ then has  compact support and $\dbar(hg)$ and $\partial(hg)$ are easily seen to be Lipschitz of order $\alpha$ with $\dbar(hg)|K=(\dbar hg+h\dbar g)|K=0$ since $h|K=1$.) Let $\zeta=x+iy$,  $f=f_1+if_2$ and $f^{\prime}(\zeta)=h_1(\zeta)+ih_2(\zeta)$, where $f_1, f_2$ and 
$h_1,h_2$ are real valued functions on $K$. It follows from (\ref{64}) and (\ref{65}) that for any $\zeta,\zeta_0\in K$ 
$$f_1(\zeta)=f_1(\zeta_0)+h_1(\zeta_0)(x-x_0)-h_2(\zeta_0)(y-y_0)+R(\zeta,\zeta_1)$$
and
$$h_j(\zeta)=h_j(\zeta_0)+R_j(\zeta,\zeta_0)\ (j=1,2),$$
where $R$ and $R_j$ are functions satisfying $|R(\zeta,\zeta_1)|\leq \kappa|\zeta-\zeta_0|^{1+\alpha}$ and 
$|R_j(\zeta,\zeta_0)|\leq\kappa|\zeta-\zeta_0|^{\alpha}$. By the Whitney extension theorem \cite[p. 177]{St} $f_1$ can be extended to a differentiable
function $g_1$ on $\bc$ such that the partial derivatives of $g_1$ are Lipschitz of order $\alpha$
and 
\begin{equation}\label{66}\frac{\partial g_1}{\partial x}=h_1,\ \ \frac{\partial g_1}{\partial y}=-h_2\ \mbox{on}\ K.\end{equation}
Similarly $f_2$ can be
extended to an appropriate function $g_2$ such that 
\begin{equation}\label{67}\frac{\partial g_2}{\partial x}=h_2,\ \ \frac{\partial g_2}{\partial y}=h_1\ \mbox{on}\ K.\end{equation}
Then $g:=g_1+ig_2$ is a required extension of $f$ since (\ref{66}) and (\ref{67}) imply that $\dbar g=0$ on $K$.
\end{proof}

In all of the above discussion in this subsection we may replace the operator $a$ by $a^{(\infty)}$ acting on $\h^{\infty}$, which
implies that the map $T_{a,g}$ defined by (\ref{611}) is completely bounded. Taking in (\ref{612}) $\xi$ and $\eta$
to be in $\h^{\infty}$, we see that
\begin{equation}\label{107}\inner{T_{a,g}(x)}{\rho}=-\frac{1}{\pi}\int_{\bc\setminus\sigma(a)}\dbar g(\zeta)
\inner{x}{(\zeta1-a)^{-1}\rho(\zeta1-a)^{-1}}\, dm(\zeta)=\inner{x}{(T_{a,g})_{\sharp}(\rho)}\end{equation}
for each $\rho=\eta\otimes\xi^*$ in the predual of $\bh$, where
\begin{align*}(T_{a,g})_{\sharp}(\rho)=-\frac{1}{\pi}\int_{\bc\setminus\sigma(a)}\dbar g(\zeta)(\zeta1-a)^{-1}\rho(\zeta1-a)^{-1}\,
dm(\zeta)\\
=-\frac{1}{\pi}\int_{\bc\setminus\sigma(a)}\dbar g(\zeta)
(\zeta1-a)^{-1}\eta\otimes((\overline{\zeta}1-a^*)^{-1})\xi)^*\, dm(\zeta).\end{align*}
A similar computation as in the proof of Lemma \ref{le6111} shows that the last integral exists and that
$\|(T_{a,g})_{\sharp}(\rho)\|\leq{\rm const.}\|\rho\|$.
Therefore we conclude that $T_{a,g}$ is weak* continuous. Further, if $S$ is any weak* continuous
$\com{(a)}$-bimodule endomorphism of $\bh$, then $S$ commutes in particular with multiplications by $(\zeta1-a)^{-1}$
and, using (\ref{107}), it follows that $S$ commutes with $T_{a,g}$. Collecting all the above results, we have proved the 
following theorem.

\begin{theorem}\label{th65}For a subnormal operator $a\in\bh$ and a function $f\in{\rm L}(1+\alpha,\sigma(a))$ 
($\alpha\in(0,1]$) let $g$ be the extension of $f$ as in Lemma \ref{le64}.
Then the map $T_{a,g}$ defined by (\ref{611}) is a central element in the algebra of all 
normal completely bounded $\com{(a)}$-bimodule endomorphisms
of $\bh$ such that $[f(a),x]=[a,T_{a,g}(x)]=T_{a,g}([a,x])$ for all $x\in\bh$. In particular the range of $d_{f(a)}$ is 
contained in the range of $d_a$ and  (\ref{598}) holds.
\end{theorem}

Now we are going to show 
that if $\sigma(a)$ is nice enough, then the Lipschitz type condition on $f$ in Theorem \ref{th65} can be relaxed:
$f$ only needs to be a Schur function. First suppose  that $\sigma(a)$ is the closed unit disc
$\weakc{\bd}$. For each $r\in(0,1)$ let $f_r(\zeta)=f(r\zeta)$. Thus each $f_r$ is a holomorphic function on a 
neighborhood $\Omega_r$ of $\weakc{\bd}$ and $f_r(a)$ can be expressed as $f(a)=\frac{1}{2\pi i}\int_{\Gamma_r}f_r(\zeta)(\zeta1-a)^{-1}\,
d\zeta$, where $\Gamma_r$ is a contour in $\Omega_r$ surrounding $\sigma(a)$ once in a positive direction.
Then for each $x\in\bh$ we have
$$[f_r(a),x]=\frac{1}{2\pi i}\int_{\Gamma_r}f(r\zeta)[(\zeta-a)^{-1},x]\, d\zeta=
\frac{1}{2\pi i}\int_{\Gamma_r}f(r\zeta)(\zeta-a)^{-1}[a,x](\zeta-a)^{-1}\, d\zeta,$$
hence
\begin{equation}\label{100}[f_r(a),x]=T_r([a,x])\ \ \mbox{and similarly}\ \
[f_r(a),x]=[a,T_r(x)]\end{equation}
where 
$$T_r(x)=\frac{1}{2\pi i}\int_{\Gamma_r}f(r\zeta)
(\zeta-a)^{-1}x(\zeta-a)^{-1}\,d\zeta.$$
If the set of completely bounded maps $T_r$ on $\bh$ ($0<r<1$) is bounded, then it has a 
limit point, say $T$,
in the weak* topology (which the space of all completely bounded maps on $\bh$ carries as a dual space, see e.g.
\cite[1.5.14 (4)]{BLM}). $T$ commutes with left and right multiplications by elements of $\com{(a)}$ 
(since all $T_r$ do). Since $f$ is continuous, $\|f_r(a)-f(a)\|\stackrel{r\to1}{\longrightarrow}0$ (this holds already if $a$ is replaced by its minimal normal extension). Then from (\ref{100}) we see that $[f(a),x]=T([a,x])$, hence
$$[f(a),x]=T([a,x])=[a,Tx]\ \ (x\in\bh).$$
These equalities hold also for $a^{(n)}$ in place of $a$ and for $x\in\matn{\bh}$, 
and  (\ref{598}) is also a consequence of $[f(a),x]=T([a,x])$. 

To estimate the norms of the maps $T_r$, let $c$ on $\k\supseteq\h$
be the unitary power dilation of $a$ (so that $a^n=pc^n|\h$ for all $n\in\bn$, where $p$ is the orthogonal
projection from $\k$ onto $\h$, see e.g. \cite{Hal} or \cite{Pa}).
Let $S_r$ be the map on $\bk$ defined by $S_r(x)=\frac{1}{2\pi i}\int_{\Gamma_r}f(r\zeta)
(\zeta-c)^{-1}x(\zeta-c)^{-1}\,d\zeta$. Then (since $\zeta-a$ is invertible if $\zeta\in\Gamma_r$) $T_r(x)=pS_r(x)|\h$
for each $x\in\bh$, where $x$ is regarded as an operator on $\k$ by setting $x|\h^{\perp}=0$. Hence $\|T_r\|\leq
\|S_r\|$.  In the special case when $c$ is
diagonal (relative to some orthonormal basis of $\k$) with eigenvalues $\lambda_i$ and $x=[x_{i,j}]$, a simple computation shows that $S_r(x)$ is represented by the matrix
$r[\frac{f(r\lambda_i)-f(r\lambda_j)}{r\lambda_i-r\lambda_j}x_{i,j}]$ (where the quotient is taken to be $f^{\prime}(r\lambda_j)$ if $\lambda_i=
\lambda_j$). Hence in this case $\|S_r\|\leq\kappa$ since $f$ is a Schur function. Since any normal operator $c$ can be approximated
uniformly by diagonal operators, it follows from the formula defining $S_r$ that the same estimate must hold for all such $c$ with $\sigma(c)\subseteq\bd$.
A similar reasoning applies also to the completely bounded norm, hence it follows that $\sup_{0<r<1}\|T_r\|_{\rm cb}<
\infty$.

Let us now consider the case when $\sigma(a)$ is the closure of its interior $U$ and $U$ is simply
connected. 
Let $h$ be a conformal bijection from $\bd$ onto $U$. If the boundary $\partial{\sigma(a)}$ of ${\sigma(a)}$ is 
sufficiently
nice, say a Jordan curve of class $C^3$, then $h$ can be extended to a bijection, denoted again by $h$, 
from $\weakc{\bd}$ onto
$\weakc{U}=\sigma(a)$, such that $h$ and $h^{-1}$ are in the class $C^2$ \cite[5.2.4]{Kr}. Then by
Theorem \ref{th65} and Proposition \ref{pr521} 
$h$ and $h^{-1}$ are Schur functions. Let
$a_0=h^{-1}(a)$. Note that $\{a_0,1\}$ generates the same Banach algebra as $\{a,1\}$ since $h$ and $h^{-1}$ can both be 
uniformly approximated by polynomials (by Mergelyan's theorem). For any Schur function 
$f$ on $\sigma(a)$
the composition $f_0:=f\circ h$ is a Schur function on $\weakc{\bd}$. (To see this, note that for any 
$\lambda\ne\mu$ in $\weakc{\bd}$ 
we may write $\frac{f(h(\lambda))-f(h(\mu))}{\lambda-\mu}=\frac{f(h(\lambda))-f(h(\mu))}{h(\lambda)-h(\mu)}
\frac{h(\lambda)-h(\mu)}{\lambda-\mu}$ and that the inequality
$\|[x_{i,j}y_{i,j}]\|_S\leq\|[x_{i,j}]\|_S\|[y_{i,j}]\|_S$ holds for the Schur norm of the Schur product of two
matrices.) Note that $f(a)=f_0(a_0)$
and  $\com{(a_0)}=\com{(a)}$.
By the previous paragraph
there exists a completely bounded $\com{(a_0)}$-bimodule map $T$ on $\bh$ such that 
$[f_0(a_0),x]=[a_0,Tx]=T([a_0,x])$, hence (using $a_0=h^{-1}(a)$ and $f_0(a_0)=f(a)$) 
\begin{equation}\label{102}[h^{-1}(a),Tx]=[f(a),x]=T([h^{-1}(a),x])\ \ \mbox{for all}\ x\in\bh.\end{equation} Now the map
$T$ is not a priori normal, but it can be replaced by its normal part $T_n$ in (\ref{102}), 
hence we may achieve that $T$ is normal. (Namely, let $T=T_n+T_s$ be the decomposition of $T$ into its normal and singular part \cite[III.2.15]{T}. This decomposition has similar properties as in the special case of linear functionals \cite[10.1.15]{KR}. Then the first equality in (\ref{102}) can be rewritten as
$[h^{-1}(a),T_nx]-[f(a),x]=-[h^{-1}(a),T_sx]$. Since the left side of the last equality is a normal function of $x$, while the right side is singular, both must be $0$. This shows that $T$ can be replaced by $T_n$ in the first equality of (\ref{102}) and a similar argument applies also to the second equality.) By Theorem \ref{th65}
there exists a completely bounded 
$\com{(a)}$-bimodule map $S$ on $\bh$ such that \begin{equation}\label{103}[a,Sy]=[h^{-1}(a),y]=S([a,y])\ \ \mbox{for all}\ y\in\bh\end{equation}
 and $S$
commutes with all normal $\com{(a)}$-bimodule maps on $\bh$ (in particular with $T$). From the first equality in (\ref{102}) and in
(\ref{103}) (with $y=Tx$)  we
have now $[f(a),x]=[h^{-1}(a),Tx]=[a,STx]$,  while from the remaining two equalities in (\ref{102}) and (\ref{103}) (with $y=x$) we deduce that $[f(a),x]=T([h^{-1}(a),x])=TS([a,x])$ for all $x\in\bh$. 
Denoting $T_{a,f}=TS=ST$, we have deduced the following theorem.

\begin{theorem}\label{th621}For a subnormal $a\in\bh$ suppose that $\sigma(a)$ is the closure of a simply connected
domain bounded by a Jordan curve
of class $C^3$. Then for 
each Schur function $f$ on $\sigma(a)$ there exists a (normal) completely bounded $\com{(a)}$-bimodule map $T_{a,f}$ on $\bh$
such that $[f(a),x]=[a,T_{a,f}(x)]=T_{a,f}([a,x])$ for all $x\in\bh$. (In particular the inequality (\ref{50}) holds for $b=f(a)$ with $\kappa=\|T_{a,f}\|$.)
\end{theorem}

In general, the Lipschitz type condition in Theorem \ref{th65} can be replaced by a similar, but less restrictive condition,
which involves a regular  modulus of continuity $\omega$
in the sense of \cite[p. 175]{St} (instead of just $\omega(t)=t^{\alpha}$) such that $\int_0^1\omega(r)/r\, dr<\infty$. 
(There exists an appropriate version of Whitney's extension theorem \cite[p. 194]{St}.)
But probably even this is too restrictive, for we do not need any requirements about 
$\partial g$ of the extension $g$ (only requirements about $\dbar g$).

\end{document}